\documentclass[11pt]{amsart}
\usepackage[pctex32]{graphics}
\usepackage{amsfonts}
\usepackage{amssymb,amscd,latexsym}
\usepackage{amsmath}
\usepackage{epsfig}
\usepackage{multirow}
\usepackage{color}
\usepackage[dvipsnames,table]{xcolor}
\usepackage{fancyvrb}   
 
\definecolor{airforceblue}{rgb}{0.36, 0.54, 0.66}
\definecolor{darkorchid}{rgb}{0.6, 0.2, 0.8}
\definecolor{darkorange}{rgb}{1.0, 0.55, 0.0}
\definecolor{darkspringgreen}{rgb}{0.09, 0.45, 0.27}
\usepackage[arrow,curve,matrix,tips,frame]{xy}
\usepackage{adjustbox}

\usepackage[utf8]{inputenc}
\usepackage[english]{babel} 
\usepackage{enumerate}
\usepackage{lscape}
\usepackage{hyperref} 
\usepackage{verbatim}
                 \hypersetup{ pdfborder={0 0 0}, 
                              colorlinks=true, 
                              linktoc=page,
                              pdfauthor={F. Russo and G. Staglian{\`o}}, 
                              pdftitle={Trisecant flop and rationality}}

\newtheorem{thm}{Theorem}[section]
\newtheorem{cor}[thm]{Corollary}
\newtheorem{prop}[thm]{Proposition}
\newtheorem{Proposition}[thm]{Proposition}
\newtheorem{defin}[thm]{Definition}

\newtheorem{lema}[thm]{Lemma}

\theoremstyle{remark}
\newtheorem{Remark}[thm]{Remark}
\newtheorem{Assumption}{Assumption}
\theoremstyle{definition}
\newtheorem{ex}[thm]{Example}
\newtheorem*{meth}{Geometric Method}

\newcommand{\PP}{{\mathbb{P}}}

\newcommand{\map}{\dasharrow}

\newcommand{\Tan}{\operatorname{Tan}}
\newcommand{\red}{\operatorname{red}}
\newcommand{\reg}{\operatorname{reg}}

\def\p{\mathbb P}

\def\I{\mathcal I}
\renewcommand{\O}{\mathcal O}
\def\rk{\operatorname{rk}}

\def\h{\operatorname{h}}
\def\axis{\operatorname{axis}}
\newcommand{\prim}{\operatorname{prim}}
\def\Sing{\operatorname{Sing}}

\newcommand{\Num}{\operatorname{Num}}

\newcommand{\Bl}{\operatorname{Bl}}

\newcommand{\Sec}{\operatorname{Sec}}

\let\calligraphy=\mathcal

\def\GG{{\mathbb G}}

\def\PP{{\mathbb P}}

\def\Oo{{\calligraphy O}}

\def\Bl{\operatorname{Bl}}

\def\Al{\operatorname{Al}}

\def\Pic{\operatorname{Pic}}

\def\Trisec{\operatorname{Trisec}}

\def\Sing{\operatorname{Sing}}

\def\Hilb{\operatorname{Hilb}}

\def\inv{^{-1}}

\let\phi=\varphi

\def\length{\operatorname{length}}

\numberwithin{equation}{section}

\begin{document}

\title[Trisecant flops and rationality]{Trisecant flops, their associated K3 surfaces and the rationality of some cubic  fourfolds}


\author[F. Russo]{Francesco Russo*}
\address{Dipartimento di Matematica e Informatica, Universit\` a degli Studi di Catania, 
Viale A. Doria 5, 95125 Catania, Italy
}
\email{frusso@dmi.unict.it  giovanni.stagliano@unict.it}
\thanks{*Both authors were partially  supported  by the PRIN 2015 {\it Geometria delle variet\`{a} algebriche}; the first author is a member of the G.N.S.A.G.A. of INDAM}

\author[G. Staglian\` o]{Giovanni Staglian\` o*}

\begin{abstract}
We provide a new construction of rationality for cubic fourfolds via Mori's theory and the minimal model program. 
As an application we present the solution of the Kuznetsov's conjecture for $d=42$ (the first open case).
Our methods also show an explicit  connection between the rationality of cubic fourfolds belonging to the first four admissible families
$\mathcal C_d$, with $d=14,26,38$, and $42$  and some birational models of K3 surfaces of degree $d$ contained in well known rational Fano fourfolds.

 \end{abstract}

\maketitle

\begin{center}{\it 
Dedicated to Antonio Lanteri on the
occasion of his seventieth birthday}
\end{center}

\section*{Introduction}

The study of the rationality of higher dimensional Fano manifolds is a very active area of research.  Many new and interesting contributions and conjectures appeared in the last decades,  mostly concerning the irrationality of very general Fano complete intersections (see for example \cite{Totaro2015,Schreieder2019,NOtte}; and also \cite{KollarHyp} and references therein).
Deep recent contributions in \cite{KontsevichTschinkelInventiones} imply  that the locus of geometrically rational fibers in a smooth family of projective manifolds is closed under specialisation, improving substantially our understanding of the loci of rational objects in the corresponding moduli spaces (see \cite{HPT} for very significant examples in dimension four). Notwithstanding,  the irrationality of the very general cubic fourfold and the complete description of the rational ones remain two of the most challenging open problems. 

A great amount of recent theoretical work on cubic fourfolds (see for example the surveys \cite{Levico, kuz2}) lead to the expectation that  the very general ones might be irrational and to the specification of  infinitely many irreducible  divisors  $\mathcal C_d$ of {\it special admissible cubic fourfolds of discriminant $d$} in the moduli space $\mathcal C$, whose union should be  the locus  of rational cubic fourfolds  ({\it Kuznetsov Conjecture}). According to this conjecture, the rationality of cubics in $\mathcal C_d$ depends on the existence of an {\it associated} K3 surface  in the sense of Hassett/Kuznetsov, see \cite{Levico, kuz2}.

The first admissible values are   $d=14, 26, 38, 42, 62, 74, 78,  86$. Our main applications of the new methods developed here will be the theoretical explanation of the role of (non minimal) K3 surfaces in determining the rationality for $d=14, 26, 38$ and 42 together with the proof  of the rationality of every cubic fourfold in $\mathcal C_{42}$ (the first open case of the conjecture) via the construction of a surface of degree nine and genus two with five nodes admitting a congruence of 8--secant twisted cubics and contained in a general cubic fourfold  in 
$\mathcal C_{42}$, see  Theorem \ref{C42_Flop}. Let us recall that Fano showed the
rationality of a general cubic fourfold in $\mathcal C_{14}$ (see \cite{Fano, BRS}), while every cubic fourfold  in the irreducible divisors $\mathcal C_{26}$ and $\mathcal C_{38}$  is rational by the   main results of \cite{RS1} (see also \cite{Explicit} and Section \ref{explicitflops} here). The proofs in \cite{RS1}  were achieved by constructing surfaces $S_d\subset \p^5$, contained in a general cubic fourfold of $\mathcal C_d$ and admitting a four dimensional family of $(3e-1)$-secant curves of degree $e\geq 2$ parametrized by a rational variety with the property  that through a general point of $\p^5$ there passes a unique  curve of the family. Then the cubics through $S_d$ become  rational sections of the universal family and  hence are birational to the rational parameter space  (see also Theorem \ref{criterion} here). 
 
This approach did not clarify the relation with the associated K3 surfaces and even the construction of explicit birational maps from general cubics $X\subset\p^5$ in $\mathcal C_{26}$ and in $\mathcal C_{38}$ to $\p^4$ (or to other notable rational smooth fourfolds $W$) in \cite{Explicit}  apparently did not provide a birational incarnation in $\p^4$ (or in $W$) of a  K3 surface of  genus 14, respectively of  genus 20, determining the linear system of the inverse map. Indeed, the  base loci of the linear systems of hypersurfaces of degree $3e-1$ having points of multiplicity $e$  along the corresponding $S_d$'s 
giving the birational maps $\mu:X\map \p^4$   are intractable reducible schemes while the base loci of the inverse maps $\mu\inv:\p^4\map X\subset\p^5$ are even worse
(see \cite[Table 2]{Explicit}). 

Here we study all these phenomena via Mori Theory and via the Minimal Model Program to explain the birational nature of the  maps $\mu:X\map W$ introduced above and of their inverses which allowed us to produce a suitable birational factorization described in diagram \eqref{diagramma1} below. This approach provides a geometric description of  the support of the base loci and, finally,  illustrates the relations of these explicit birational maps with the K3 surfaces {\it associated} to $X$, see Subsections \ref{assK3C38} and \ref{TF42}. Last but not least, computer-aided methods 
play a central role in some key points, also due to the complexity of the geometry involved.

Our method starts with the observation that many of the known examples of surfaces $S_d\subset\p^5$ used to describe the special divisors $\mathcal C_d$ for small $d$ (not only the  admissible ones) have ideal generated by cubic forms defining
a map  $\phi:\p^5\map Z\subset\p^N$ which is birational onto its image. The restriction to a general cubic $X$ through $S_d$ defines a birational map $\phi:X\map Y\subset\p^{N-1}$ with $Y$ a general hyperplane section of $Z$. In many cases, the birational morphism $\tilde\phi:X'=\Bl_{S_d}X\to Y$  is a small contraction, whose exceptional locus consists of a (union of) smooth surface(s) $T'\subset X'$ ruled by the strict transforms of trisecant lines to $S_d$. Since $K_{X'}$ is zero on the strict transforms of trisecant lines, the map $\tilde\phi$ is a flop small contraction. In Theorem \ref{trisflop} we show that under the previous hypothesis there exists a flop $\tilde\psi:W'\to Y$ of the surface $T'\subset X'$ with $W'$ a smooth projective fourfold. This {\it trisecant flop} $\tau:X'\map W'$ is constructed by analyzing the  splitting of $N^*_{T'/X'}$ restricted to a general strict transform of a trisecant line to $S_d\subset\p^5$,
see Remark \ref{gentris}. The existence of a congruence of $(3e-1)$-secant rational curves to $S_d$ of degree $e\geq 2$ produces an extremal ray on $W'$   with divisorial locus, giving  a birational morphism $\nu:W'\to W$ with $W$ a $\mathbb Q$-factorial Fano variety having $\Pic(W)\simeq \mathbb Z$, see Theorem~\ref{contraction}. 
The birational morphism $\nu$ is (generically) the  blow-up of an irreducible surface $U\subset W$, 
which for the  admissible cases $d=14, 26, 38, 42$ is a birational incarnation of the associated K3 surface to $X$. 
Moreover, the map  $\mu=\nu\circ\tau\circ\lambda^{-1}:X\map W$ is given by a linear
system of forms of degree $3e-1$ with points of multiplicity at least $e$ along $S_d$,
while $\mu\inv:W\map X$ is given by a linear system of divisors in $|\Oo_W(e\cdot i(W)-1)|$ 
having points of multiplicity at least $e$ along $U$, where $i(W)$ is the index of $W$. Everything is captured by the following diagram, where $R'\subset W'$ is the strict transform on $W'$ of the locus $R\subset W$ of $i(W)$-secant lines to $U$:
\begin{equation}\label{diagramma1}
 \UseTips
 \newdir{ >}{!/-5pt/\dir{>}}
 \xymatrix{
 &\Bl_{T'}X'=\Bl_{R'}W'\ar[ld]_\sigma\ar[rd]^{\omega}&\\             
 X' \ar@{-->}[rr]^{\tau}\ar[rd]^{\tilde\phi}\ar[d]_\lambda &                             &W'\ar[ld]_{\tilde\psi}\ar[d]^\nu\\
X\ar@{-->}[r]^{\phi}\ar@{-->}@/_0.6cm/[rr]_{\mu} &Y& W \ar@{-->}[l]_{\psi}    }
\end{equation}
In the explicit examples with $d=14,26,38,42$ considered here, diagram \eqref{diagramma1} together with  some standard computations of the middle cohomology of fourfolds  under the blow-up of smooth surfaces (see for example the discussion in \cite[pages 45--46]{Levico})  shows that the (possibly singular non minimal) K3 surface $U\subset W$  is a birational incarnation of the K3 surface {\it associated} to $X$ via Hodge Theory or, equivalently, via Derived Category (see Subsection  \ref{assK3C38} for the complete analysis of the case $d=38$).

The above theoretical results (and all the examples we have constructed) point out that, for low values of the discriminant $d$,  the {\it birational association}  between admissible cubic fourfolds and K3 surfaces  passes through the construction of very special (and in many cases also singular) non minimal birational models of these surfaces in the fourfolds $W$ by mean of  peculiar linear systems of hyperplane sections (often with base points of high multiplicities) on the associated K3 surfaces, exactly as in the case $d=14$ considered by Fano. In particular, we are  able to prove the rationality of every cubic fourfold in $\mathcal C_d$ for $d=14,26, 38, 42$ via  trisecant flops and to show that their  inverses are determined by linear systems having a prescribed multiplicity along (non minimal) birational models of the associated K3 surfaces.

 Some of these examples of non minimal K3 surfaces have been also studied by Voisin in \cite[\S 4]{Voisin} to prove  vanishing results related to Lehn's Conjecture and have been later  considered by Fontanari and Sernesi in \cite[Theorem 10]{FS}.   Our constructions of K3 surfaces are based on a different geometric method developed in Section \ref{geomet} and on computations, which among other things also provide explicit equations for
 the general K3 surfaces of genera $8$, $14$, $20$, and $22$. 

Since  the Kodaira dimension of $\mathcal C_d$ is nonnegative  for $d\geq 86$ admissible (see \cite{TaVaA} and also \cite[Proposition 1.3]{Nuer}), 
a description of $\mathcal C_d$ via surfaces $S_d$'s with at least two cubics containing them as the one considered here (and also in \cite{Nuer, RS1, Explicit})  is not allowed anymore in this range because this would imply the uniruledness of $\mathcal C_d$ (see the argument in the proof of Theorem  \ref{C42S42}). The admissible values $d=62, 74, 78$ are  the last ones for which the elements in  $\mathcal C_d$ may be arranged into homaloidal linear systems of cubics through an irreducible surface $S_d$, simplifying the construction of $U$ and $W$ via the trisecant flop. For  $d\geq 86$ admissible one needs  to consider suitable generalisations of this approach, see  Remark \ref{hypothesis}, increasing substantially the difficulty of the conjecture.

The techniques introduced here open  the way to further applications of this circle of ideas  to prove the rationality of other classes of Fano fourfolds, see \cite[Section 4]{Explicit} and \cite{HS}.
\vskip 0.2cm

{\bf Acknowledgements}.  This project started from an intuition
of J\' anos Koll\' ar that  the maps $\mu:X\map W\subseteq\p^N$ might be useful to describe  the congruences (see \cite[Subsection~1.1]{Explicit}, Subsection \ref{nec}  here, and
\cite[Section 5, \S 29]{KollarHyp}). We are very indebted
to him for the suggestion.
We wish to thank Michele Bolognesi for  his continuous support 
along the years and for many useful conversations on these  topics.
We are also very grateful  to Sandro Verra for several discussions on the  subject and to Brendan Hassett  for his remarks on  a preliminary version of the manuscript. We thank the referees for their careful reading and for their useful comments. 
\section{Preliminaries}

\subsection{Congruences of \texorpdfstring{$(3e-1)$}{(3e-1)}-secant curves of degree \texorpdfstring{$e$}{e} to surfaces in \texorpdfstring{$\p^5$}{P5}}\label{nec}

The following definitions have been introduced in \cite[Section 1]{RS1}. Let $\mathcal H$ be 
an irreducible proper family of  (rational or of fixed arithmetic genus) curves of degree $e$ in $\p^5$ whose general element is irreducible.
We have a diagram
\begin{equation}\label{diagramma congruenza}
\xymatrix{
 \mathcal D\ar[dr]^{p}\ar[d]_{\pi}&\\
\mathcal H&\p^5,}
\end{equation}
where $\pi:\mathcal D\to \mathcal H$
is the  universal family over $\mathcal H$ and where  $p:\mathcal D\to\p^5$ is  the tautological morphism. 
Suppose moreover that $p$ is birational and that a general member $[C]\in \mathcal H$ is ($re-1$)-secant to an irreducible surface $S\subset\p^5$,
that is $C\cap S$ is a length $r e-1$ scheme, $r\in\mathbb N$. We shall call such a family $\mathcal H$ (or $\mathcal D$ or $\pi:\mathcal D\to\mathcal H$) a {\it congruence of }  ($re-1$)-{\it secant curves of degree $e$ to $S$}.  Let us remark that necessarily $\dim(\mathcal H)=4$.
\medskip

An irreducible  hypersurface $X\in|H^0(\mathcal I_{S}(r))|$ is said to be {\it transversal to the congruence $\mathcal H$} if  the unique curve of the congruence passing through a general point  $p\in X$ is not contained in $X$. A crucial result is the following.

\begin{thm}\label{criterion} {\rm \cite[Theorem 1]{RS1}} Let $S\subset\p^5$ be a surface admitting a congruence of {\rm(}$re-1${\rm)}-secant curves of degree $e$ parametrized by $\mathcal H$.
If  $X\in|H^0(\mathcal I_{S}(r))|$ is an irreducible hypersurface transversal to $\mathcal H$,  then $X$ is birational to $\mathcal H$.

If the map   
$\Phi=\Phi_{|H^0(\mathcal I_{S}(r))|}:\p^5\map \p(H^0(\mathcal I_{S}(r)))$
is birational onto its image, then a general  hypersurface  $X\in|H^0(\mathcal I_{S}(r))|$  is birational to $\mathcal H$.

Moreover, under the previous hypothesis on $\Phi$, if  a general  element in $|H^0(\mathcal I_{S}(r))|$ is smooth, then every $X\in |H^0(\mathcal I_{S}(r))|$ with at worst rational singularities is birational to $\mathcal H$.
\end{thm}

Since $p:\mathcal D\to\p^5$ is birational, we also have a rational map $$\varphi=\pi\circ p^{-1}:\p^5\map\mathcal H,$$ whose  fiber through a general $p\in\p^5$, $F=\overline{\varphi^{-1}(\varphi(p))}$, is the unique curve of the congruence passing through $p$. 

It is natural to ask what linear systems on $\p^5$ give the abstract birational maps $\varphi:\p^5\map \mathcal H$ as above
or  their restrictions to a general $X\in |H^0(\mathcal I_S(r))|$. The
 linear system  $|H^0(\mathcal I_S^e(re-1))|$, when not empty,  contracts the fibers of $\varphi$ and in \cite{Explicit} we showed that, quite surprisingly, in many cases they can provide a birational geometric realization of $\varphi$  for $r=3$, yielding birational maps from cubic hypersurfaces through $S$ to  $\mathcal H$ with $\mathcal H=\p^4$ or with $\mathcal H$ a notable Fano fourfold.
In the sequel  we shall develop a theoretical framework for these phenomena in order to be able to understand also the birational maps defined by the previous linear systems.

\subsection{Divisorial contractions, small contractions and flops}\label{divisorial}
We introduce some general  definitions of the 
Minimal Model Program (MMP for short), adapting them to our setting. 
\medskip

Let $X$ be a smooth projective irreducible fourfold defined over the complex field with $\rho(X)=1$ (here $\rho(X)$ denotes
the Picard number of $X$) and let $\phi:X\map W$ be a birational map onto a smooth (or at least $\mathbb Q$--factorial) irreducible projective fourfold, whose base locus scheme contains a surface $S$ with at most a finite number of nodes.

Let $\lambda:X'=\Bl_SX\to X$ be the blow-up of $S$ and consider the  diagram:

\begin{equation}\label{diag00}
\UseTips
 \newdir{ >}{!/-5pt/\dir{>}}
 \xymatrix{
 \Bl_SX\ar[d]_\lambda\ar@{-->}[rrd]^{\tilde\phi}&                              \\
X\ar@{-->}[rr]_{\phi}& &W. 
}
\end{equation}

When $W$ is smooth, the complexity of the birational map $\phi:X\map W$ depends on the base locus scheme of $\tilde\phi:\Bl_SX\map W$.
Surely the easiest case to be considered is when $\tilde\phi:\Bl_SX\to W$ is a morphism, that is $\phi$ is a {\it special birational map} in the sense of  Semple and Tyrrell
(solved by a single blow-up along a smooth irreducible variety).

If $X\subset\p^5$ is a cubic fourfold and if $S\subset X$ is smooth,
few examples of special birational maps of the above type exist.
Two examples of maps of this kind were firstly considered by  Fano in \cite{Fano}, have been revisited in modern terms in \cite{AR,BRS} and played a fundamental role in the formulation of Kuznetsov Conjecture.
\begin{ex}\label{esempiFano}
Letting  $\phi:X\map W$ be a special birational map with $X\subset\p^5$ a cubic fourfold, letting $B\subset W$ be the base locus scheme of $\phi\inv$ and letting  $U=B_{\red}$, Fano's examples  are the following:
\begin{enumerate}[(i)]
\item\label{primoEs} $S\subset\p^5$ is a smooth quintic del Pezzo surface, $W=\p^4$, $\phi$ is given by $|H^0(\mathcal I_S(2))|$  and $U\subset\p^4$ is a surface of degree
9 and sectional genus 8 having at most a finite number of singular points corresponding to planes in $X$ spanned by conics in $S$.
If non singular, the surface $U$ is the projection from a 5-secant $\p^3\subset\p^8$ of a smooth K3 surface of degree 14 and genus 8 and $\phi\inv$ is given by  $|H^0(\mathcal I_U(4))|$.
\item $S\subset\p^5$ is a smooth quartic rational normal scroll, $W=Q\subset\p^5$ is a smooth quadric hypersurface, $\phi$ is given by $|H^0(\mathcal I_S(2))|$ and $U\subset\p^5$ is a surface of degree
10 and sectional genus 8 having at most a finite number of singular points corresponding to planes in $X$ spanned by conics in $S$.
If non singular, the surface $U$ is the projection from the tangent plane of  a smooth K3 surface of degree 14 and genus 8 and $\phi\inv$ is given by  $|H^0(\mathcal I_U(3))|_{|_Q}$.
\end{enumerate}
\end{ex}

\begin{Remark} The two surfaces $S\subset\p^5$ appearing in Example~\ref{esempiFano} are the only smooth surfaces in $\p^5$ admitting a congruence of secant lines ($r=3$ and $e=1$ in the definition), see for example \cite{OADP}. The lines
of the congruence contained in $X$ describe the exceptional locus $\overline E$ of $\tilde\phi$ (or equivalently the exceptional locus $\lambda(\overline E)$ of $\phi:X\map W$) and  are birationally parametrized by the surfaces $U\subset W$.
\end{Remark}

The general MMP philosophy suggests that meaningful birational properties of (rational) cubic fourfolds might be related to small contractions from $X'$.
So one can start to investigate birational properties of cubic fourfolds from the point of view of the MMP and, when they exist, to consider the most elementary links in the Sarkisov Program associated to small contractions, i.e.  flops and flips (one may consult  \cite{HMK} for results about this program in arbitrary dimension).

\begin{defin}{\rm
Let $X$ be a smooth  irreducible projective  variety (from now on a projective manifold) and let
$\tilde\phi:X\to Y$ be  a {\it small contraction}, i.e. $\tilde\phi$ is a birational morphism onto a normal variety $Y$ inducing an isomorphism in codimension one
and such that  $\rho(X/Y)=1$.

If  $K_{X}\cdot C=0$ for every irreducible curve contracted by $\tilde\phi$, then $\tilde\phi:X\to Y$ is called {\it a small flop contraction}.
A small flop contraction $\tilde\psi: W\to Y$ with $W$ a projective manifold 
is called a {\it flop of $\tilde\phi$}.

The resulting birational
map $\tau=\tilde\psi\inv\circ\tilde\phi:X\map W$ is usually called a {\it flop} if it is not an isomorphism. Since we assume $\rho(X/Y)=1=\rho(W/Y)$,  given $\tilde\phi$ one can prove that  the morphism $\tilde\psi$, if it exists,  is unique as soon as $\tau$ is not an isomorphism.}
\end{defin}

One can {\it flop} the small contraction $\tilde\phi:X\to Y$ by constructing a projective manifold $V$ and two birational morphisms $\sigma:V\to X$ and $\omega:V\to W$
such that $\sigma^*(K_X)=\omega^*(K_W)$. This means that  the exceptional locus of $\sigma$, which is divisorial by the smoothness of $X$,  is contracted by $\omega$ and that we have a commutative diagram: 
\begin{equation}\label{diag0}
\UseTips
 \newdir{ >}{!/-5pt/\dir{>}}
 \xymatrix{
 &V\ar[ld]_\sigma\ar[rd]^{\omega}&\\            
 X \ar@{-->}[rr]_{\tau}\ar[rd]_{\tilde\phi}&                             &W\ar[ld]^{\tilde\psi}\\
&Y&
}
\end{equation}

First of all one may ask if there exist flops of this kind on the fourfolds $X'=\Bl_SX$ obtained from cubic fourfolds $X\subset\p^5$ by blowing-up a mildly singular surface $S\subset X$.
 As we shall see this is the case under some hypothesis and this occurrence is deeply related to the rationality of some {\it special} cubic fourfolds (or of other {\it special} fourfolds).

\subsection{Condition \texorpdfstring{$\mathcal K_3$}{K3} and examples of small contractions on cubic fourfolds}\label{sbsK3}
Let us recall that, given homogeneous forms $f_i$ of degree $d_i\geq 1$, $i=0,\ldots,M$, a vector of homogenous forms $(g_0,\ldots,g_M)$
is {\it a syzygy} if $\sum_{i=0}^Mf_ig_i=0$. If $d_1=\cdots=d_M=d$ and if $\deg(g_i)=h$ for every $i=0,\ldots, M$, then  we say that $(g_0,\ldots, g_M)$
is a syzygy of degree $h$ and for $h=1$ we shall say that the syzygy is {\it linear}. For $i<j$ the syzygies $(0,\ldots,0,f_j,0,\ldots,0, -f_i,0,\ldots0)$,
corresponding to the trivial identity $f_if_j+f_j(-f_i)=0$ are called {\it Koszul syzygies}. We say that the Koszul syzygies are generated by 
the linear ones if they belong to the submodule generated by the linear syzygies. This is the {\it condition $\mathcal K_d$} introduced by Vermeire in  \cite{Vermeire}.

The next result provides  a wide class of examples of rational maps with linear fibers (hence birational under mildly natural geometrical assumptions on their base locus scheme).

\begin{prop}\label{verme} {\rm (\cite[Proposition 2.8]{Vermeire}\footnote{The hypothesis on the absence of lines in the base locus scheme of $\phi$ appearing  in the original version is not necessary (see the first arXiv versions of \cite{BRS} for a proof).})} Let $f_0,\ldots, f_M$ be homogeneous forms in $N+1$ variables of degree $d\geq 2$ satisfying condition $\mathcal K_d$.Then the closure of each fiber of the rational map
$$\phi=(f_0:\cdots:f_M):\p^N\map\p^M$$
is a linear space $\p^s$. For $s>0$  the closure of the fiber intersects scheme theoretically the base locus scheme of $\phi$ along a hypersurface of degree $d$.
\end{prop}

\begin{Remark}\label{linsyz}
Suppose that an irreducible surface $S\subset\p^5$ is scheme-theoretically defined by cubic equations satisfying  condition $\mathcal K_3$. Then, by Proposition \ref{verme},  every positive dimensional fiber of $\phi:\p^5\map Z$ is a linear space $\p^s$ cutting $S$
in a cubic hypersurface $S\cap \p^s$ if $s>0$. In particular $0\leq s\leq 2$ (except some trivial cases) and $s=2$ occurs only for planes  spanned by cubic curves contained in  $S$, which are mapped to a point by $\phi$. Hence if condition $\mathcal K_3$  for $S\subset\p^5$ holds and if a general cubic $X\subset\p^5$ through $S$ does not contain any plane spanned by cubic curves on $S$,  the exceptional locus $T\subset X$ of the restriction of $\phi$ to $X$ is ruled by proper trisecant lines. As we shall see in Section \ref{hilbtris}  the expected dimension of $T$ is two
so that  surfaces in $\p^5$ defined by cubic equations satisfying condition $\mathcal K_3$ may  naturally produce examples of small contractions on $X'=\Bl_SX$.
\end{Remark}

\section{The trisecant flop and the extremal congruence contraction} \label{s2}

We first introduce and study the behaviour of trisecant lines to a {\it general}  non degenerate irreducible projective surface $S\subset\p^5$.

\subsection{The Hilbert scheme of trisecant lines  to \texorpdfstring{$S\subset\p^5$}{S in P5}}\label{hilbtris} For the generalities we shall follow the treatment in \cite{Bauer}.
Let $\Hilb^r\p^5$ (respectively $\Hilb^r S$) be the Hilbert scheme of 0--dimensional length $r\geq 2$ subschemes of $\p^5$ (respectively of $S\subset\p^5$) and let $\Hilb^r_c\p^5\subset\Hilb^r\p^5$ be the open non-singular subscheme consisting of {\it curvilinear} length $r$ subschemes,
that is length $r$ subschemes which, locally around every point of their support, are contained in a smooth curve of $\p^5$. We can define $\Hilb^r_cS$ as the scheme--theoretic intersection between $\Hilb^rS$ and $\Hilb^r_c\p^5$ inside $\Hilb^r\p^5$. 

Let $\Al^r\p^5\subset\Hilb_c^r\p^5$ denote the subscheme consisting of
{\it aligned subschems of length $r$}, that is subschemes of length $r$ contained in a line. Finally,
the Hilbert scheme of length $r$ aligned subscheme of $S$, denoted by $\Al^rS$, is the scheme-theoretic intersection of  $\Al^r\p^5$ with $\Hilb^r_cS.$

The schemes $\Hilb_c^r\p^5$ and $\Al^r\p^5$ are smooth of dimension
$5r$ and $8+r$, respectively. Moreover, if $S\subset\p^5$ is smooth, then $\Hilb_c^rS$ is smooth of dimension $2r$. In particular, either $\Al^3S$ is empty or every irreducible components of $\Al^3S$ has dimension at least  
$$\dim(\Al^3\p^5)+\dim(\Hilb_c^3S)-\dim(\Hilb_c^3\p^5)=2,$$ which is therefore the {\it expected dimension} of $\Al^3 S$. So, for an irreducible projective surface $S\subset\p^5$, one might expect that, with few exceptions, the Hilbert scheme $\Al^3 S$ of trisecant lines  is  of pure dimension two.

There exists a natural morphism of schemes 
$$\axis: \Al^rS\to \mathbb G(1,5),$$
sending each length $r\geq 2$ aligned subscheme of $S$ to the unique line containing its support, that is to the multisecant line to $S$ determined by the subscheme of points (counted with multiplicity). 
Let $q:\mathcal L\to \mathbb G(1,5)$ be the universal family and let $p:\mathcal L\to\p^5$ be the tautological morphism. Then 
$$\Trisec(S):=p(q^{-1}(\axis(\Al^3S)))\subset\p^5$$
is called the {\it trisecant locus of} $S\subset\p^5$.
The previous count of parameters and analysis show: that the expected dimension of $\Trisec(S)$ is three; that every irreducible component of $\Trisec(S)$ has dimension at least two; that the irreducible components of dimension two of $\Trisec(S)$ are either $S$ (in this case $S$ is ruled by lines) or planes cutting $S$ along a plane curve of degree at least three, see \cite{Bauer}. 

By the Trisecant Lemma, see \cite[Proposition 1.4.3]{Umi}, a general secant line to an irreducible non degenerate surface $S\subset\p^5$ is not a trisecant line. So $\dim(\Al^3 S)\leq 3$ and  $\dim(\Trisec(S))\leq 4$. Very few examples of irreducible non degenerate surfaces $S\subset\p^5$ having $\dim(\Trisec(S))=4$ are known,  most of them are very singular (see  \cite{Rogora} for a description) but a complete classification is still lacking. The smooth irreducible non degenerate surfaces $S\subset\p^5$
with $\dim(\Trisec(S))\leq 2$ are classified in \cite{Bauer}. 

In our analysis we shall always consider the most general case $\dim(\Trisec(S))=3$. While the condition on the dimension is expected by the above parameter count, the generic smoothness of an irreducible component of $\Al^3S$ is
related to the dimension of the corresponding locus and to the tangential behaviour of $S\subset\p^5$ at the points of intersection of a general trisecant line by \cite[Proposition 4.3]{GrusonPeskine} (see also \cite[Section 1]{CC} and  \cite{Ran} for spectacular generalisations). We shall specialise  this general result to our setting.

\begin{prop}\label{Al3S} Let $S\subset \p^5$ be an irreducible projective surface and let $L\subset\p^5$ be a  proper trisecant line to $S$
such that $L\cap S=\{p_1, p_2, p_3\}$, with $p_1, p_2, p_3$ distinct smooth points of $S$, and with $[L]$ belonging  to an  irreducible component $A$ of $\Al^3S$ of dimension two. Then $\Al^3S$ is smooth at $[L]$ if and only if the tangent planes to $S$ at the points $p_i$'s are in general linear position, that is $T_{p_j}S\cap T_{p_k}S=\emptyset$ for any distinct $p_j, p_k\in  L\cap S$. In particular, if this condition holds at the point $[L]\in A$, then  $ A$ is generically smooth  and for a general $[L]\in A$ the tangent planes at the points $p_i$'s are in general linear position.

Moreover, in this case the irreducible component of $\Trisec(S)$ corresponding to $A$ has dimension 3 and through a general point $q$ of
this irreducible component there passes a finite number of trisecant lines to $S$, which are smooth points of the zero dimensional Hilbert scheme of  trisecant lines to $S$ passing through $q$.
\end{prop}

This result and the previous analysis motivate the next definition.

\begin{defin}{\rm ({\bf Expected trisecant behaviour}) Let $S\subset\p^5$ be an irreducible non degenerate projective surface. If $\dim(\Trisec(S))=3$ and if every irreducible component of $\Al^3S$  of dimension two, whose trisecant lines describe an irreducible component of $\Trisec(S)$ of dimension three, is  generically reduced (and hence generically smooth), then $S\subset\p^5$ is said to have the {\it expected trisecant behaviour.}

A trisecant line to an irreducible surface $S\subset\p^5$ having the expected trisecant behaviour is said to be {\it general} if it is the general element of
an irreducible component of $\Al^3 S$ whose locus has dimension three.}
\end{defin}

We now start to study the consequences of this natural condition. 

\begin{lema}\label{split1} Let $S\subset\p^5$ be an irreducible non degenerate projective surface with the expected trisecant behaviour and with at most a finite number of singular points, let $L\subset \p^5$ be a general trisecant line and let $L'\subset\Bl_S\p^5$ denote the strict transform of $L$.
Then
$$N_{L'/\Bl_S\p^5}\simeq \Oo_{\p^1}\oplus \Oo_{\p^1}\oplus \Oo_{\p^1}(-1)\oplus \Oo_{\p^1}(-1).$$
 If $ \tilde T\subset\p^5$ denotes the unique irreducible component of $\Trisec(S)$ containing $L$ and if $\tilde T'$ denotes
the strict transform of $\tilde T$ on $\Bl_S\p^5$, then
$$N_{L'/ \tilde T'}\simeq \Oo_{\p^1}\oplus \Oo_{\p^1}.$$
Furthermore, if $\tilde T'$ is smooth along $L'$, then
$$N_{\tilde T'/\Bl_S\p^5_{|L'}}\simeq \Oo_{\p^1}(-1)\oplus \Oo_{\p^1}(-1).$$
\end{lema}
\begin{proof}  Let $S_{\reg}=S\setminus\Sing(S)$ be the locus of smooth points of an irreducible non degenerate surface $S\subset\p^5$ and let $\pi:\Bl_S\p^5\to\p^5$ be the blow-up of $\p^5$ along $S$. Then $\pi\inv(\p^5\setminus\Sing(S))$ is a smooth variety. 
Since $\Sing(S)$ is zero dimensional,  a general $[L]\in\Al^3S$  will cut
$S$ in three smooth distinct points and $L'$ will be contained in the smooth locus of $\Bl_S\p^5$. In particular, the normal bundle $N_{L'/\Bl_S\p^5}$ is locally free
of rank four .

The strict transforms of general trisecant lines to $S$ determine a proper family of dimension two of smooth rational curves on $\Bl_S\p^5$ and the curve $L'$ represents a smooth point of this family by hypothesis,  yielding 
$N_{L'/\Bl_S\p^5}\simeq \Oo_{\p^1}\oplus \Oo_{\p^1}\oplus \Oo_{\p^1}(-1)\oplus \Oo_{\p^1}(-1)$.
Indeed,  $\deg(N_{L'/\Bl_S\p^5})=-2$ by Adjunction Formula, $h^0(N_{L'/\Bl_S\p^5})=2$ by the generic smoothness hypothesis of the irreducible component to which $L'$ belongs, and $h^0(N_{L'/\Bl_S\p^5}(-1))=0$ since through a general point of $\tilde T'$ there passes a finite number of curves of the family by the last part of Proposition \ref{Al3S}.   Then the exact sequence
\begin{equation}\label{N1}
0\to N_{L'/ \tilde T'}\to N_{L'/\Bl_S\p^5}\to N_{\tilde T'/\Bl_S\p^5_{|L'}}
\end{equation}
assures that $N_{L'/ \tilde T'}$ is torsion free and hence locally free of rank 2. Moreover, letting $N_{L'/\tilde T'}\simeq \Oo_{\p^1}(a_1)\oplus \Oo_{\p^1}(a_2)$, we deduce $a_i\leq 0$ for $i=1,2$. Since the curve $L'$ moves in a family of dimension two inside $\tilde T'$, we have $h^0(N_{L'/\tilde T'})\geq 2$ and hence $a_1=a_2=0$.
Finally, if
 $\tilde T'$ is smooth along $L'$, then the exact sequence \eqref{N1} is also exact on the right and 
$N_{\tilde T'/\Bl_S\p^5_{|L'}}\simeq \Oo_{\p^1}(-1)\oplus\Oo_{\p^1}(-1)$.
\end{proof}

\begin{Remark}\label{gentris} We are interested in studying the birational properties of  smooth cubic hypersurfaces $X\subset\p^5$ passing through an irreducible projective surface $S\subset\p^5$ having the expected trisecant behaviour and with at most a finite number of singular points. Retain the notation of Lemma \ref{split1} and suppose $L\subset X\subset\p^5$ is a general proper trisecant line to $S$ contained in $X$. Since $N_{\Bl_SX/\Bl_S\p^5_{|L'}}\simeq\Oo_{\p^1}$, since  $N_{L'/\Bl_S\p^5}\simeq \Oo_{\p^1}\oplus \Oo_{\p^1}\oplus \Oo_{\p^1}(-1)\oplus \Oo_{\p^1}(-1)$ and since we have the exact sequence:
$$0\to N_{L'/\Bl_SX}\to N_{L'/\Bl_S\p^5}\to N_{\Bl_SX/\Bl_S\p^5_{|L'}}\to 0,$$
we deduce that either $N_{L'/\Bl_SX}\simeq \Oo_{\p^1}\oplus \Oo_{\p^1}(-1)\oplus \Oo_{\p^1}(-1)$ or  $N_{L'/\Bl_SX}\simeq \Oo_{\p^1}\oplus \Oo_{\p^1}\oplus \Oo_{\p^1}(-2)$.
If the  last splitting holds, then:  either the family of strict transforms of trisecant lines to $S$ contained in $X$ to which $L'$ belongs  is two dimensional (that is  $\Trisec(S)\subseteq X$) and is generically smooth;
 or  this family  is one dimensional but not generically reduced.  If $X\subset\p^5$ does not contain $\Trisec(S)$ and if $S$ has the expected trisecant behaviour, then the family of trisecant lines to $S$ contained in $X$ is one dimensional and the corresponding  locus  has dimension two.

Thus when  $X$ is {\it sufficiently general},  when $S$ has the expected trisecant behaviour and at most a finite number of singular points,  the locus of trisecant lines to $S$ contained in $X$ is of pure dimension two and one expects that the one dimensional families of trisecant lines to $S$ contained in $X$ are generically smooth as subschemes of the corresponding parameter space.

The previous natural expectation/hypothesis translates into the following conditions, letting notation be as above:
$$N_{L'/\Bl_SX}\simeq \Oo_{\p^1}\oplus \Oo_{\p^1}(-1)\oplus \Oo_{\p^1}(-1).$$
If $T\subset X$ denotes the unique two dimensional irreducible component of the locus of trisecant lines to $S$ contained in $X$ to which $L$ belongs, then
$$N_{L'/T'}\simeq \Oo_{\p^1}.$$ Furthermore, if $T'$ is smooth along $L'$, then
\begin{equation}\label{splitT'}
N_{T'/\Bl_SX_{|L'}}\simeq \Oo_{\p^1}(-1)\oplus \Oo_{\p^1}(-1).
\end{equation}

Condition \eqref{splitT'} is crucial. Indeed, as we shall see in the next section, it essentially says that $T'$ can be flopped producing another four dimensional variety birational to $\Bl_SX$ and hence to $X$ in a very natural way. 

 If an irreducible projective surface $S\subset\p^5$ has the expected trisecant behaviour and if $S$ satisfies condition $\mathcal K_3$, then, for a general cubic through $S$, the expected splittings listed above hold for a general proper trisecant line to $S$ contained in the cubic, see the proof of Theorem \ref{trisflop}. There are also many other examples of different flavour for which the above conditions naturally hold and which naturally lead to flops of the trisecant locus contained in the cubic fourfold.

\end{Remark}

\subsection{Assumptions and main definitions}\label{assdef}

\begin{Assumption} \label{assumption1} 
  Suppose we have a smooth irreducible projective surface (the treatment can be  extended to surfaces with at most a finite number of singular points)   $S\subset\p^5$, scheme-theoretically defined by cubic hypersurfaces
and such that the associated rational map
$$\phi:\p^5\map\p(H^0(\I_S(3))=\p^N$$
is birational onto the closure of its image $Z=\overline{\phi(\p^5)}\subset\p^N.$ 
\end{Assumption}
Then
the restriction of $\phi$ to a general $X\in|H^0(\I_S(3))|$ 
induces a birational map
$$\phi:X\map Y\subset\p^{N-1}$$
with $Y$ the corresponding hyperplane section of $Z\subset\p^N$. 
On $X'=\Bl_SX$ we have
$$-K_{X'}=\lambda^*(-K_X)-E=3H'-E$$
and our hypothesis on the defining equations of $S$ and on the birational map $\phi:\p^5\map Z$ can be reformulated by saying that
 $-K_{X'}$ is a big divisor generated by its global sections. In particular, $-K_{X'}$ is nef and big.

The induced morphism 
$$\tilde\phi:\Bl_SX\to Y$$
 is  a small contraction (with very few exceptions).
Indeed, the base locus scheme of $\phi$ is the surface $S$ and $\phi$ contracts any irreducible (rational) curve $C\subset X$ of degree $e\geq 1$ which is $3e$--secant to $S$, i.e. such that $\length(C\cap S)=3e$ ({\it proper} $3e$--secant curve to $S$). Let us indicate
by $T\subset X$ the closure of the locus of proper $3e$-secant curves  to $S$ contained in $X$. If $L'\subset X'$ is the strict transform of a proper trisecant line to $S$ contained in $X$, let $[L']$ denote its numerical class in $N_1(X')$.

The strict transform $C'\subset X'$ of a proper $3e$--secant curve $C\subset X$ to $S$ of degree $e\geq 1$ satisfies $[C']=[eL']$.
Therefore on $X'=\Bl_SX$ we have 
 $$K_{X'}\cdot C'=(E-3H')\cdot C'=3e-3e=0$$ 
 for  curves $C'\subset X'$ as above.

\begin{defin}{\rm ({\bf Trisecant flop}) Let notation and assumptions be as above. If $\tilde\phi:X'\to Y$ is a small contraction of  curves in $\mathbb R[L']$,
then it is called a {\it trisecant  flop contraction}.
If $\tilde\phi: X'\to Y$ is a trisecant  flop contraction and if there exists a flop  $\tilde\psi:W'\to Y$ of $\tilde\phi$ with $W'$ a projective fourfold, then the resulting birational
map $\tau:X'\map W'$ will be called the {\it trisecant flop} (of $\tilde\phi:X'\to Y$).} 
\end{defin}

Let us remark that, by definition, if $\tilde\phi:X'\to Y$ is a trisecant flop contraction, then the exceptional locus of $\tilde\phi:X'\to Y$ has dimension at most two and the irreducible components
of dimension two are covered by proper $3e$--secant (rational) curves (in most cases they are ruled by these curves). By Zariski's Main Theorem, a positive dimensional fiber is connected so that a general positive dimensional fiber is smooth and irreducible.

During our study of birational maps  $\phi:\p^5\map Z$ of the type described above, we constructed many surfaces $S\subset\p^5$ inducing trisecant flop contractions on a general cubic fourfold $X$ through $S$.
For example surfaces satisfying condition $\mathcal K_3$ but not only (see Table~\ref{tabella}).

\subsection{Existence of the trisecant flop}
For simplicity we shall now assume as above that $S$ is smooth.
As always, let $\lambda: X'=\Bl_SX\to X$ be the blow-up
of $X$ along $S$, let $E\subset X'$ be the exceptional divisor and let $H'=\lambda^*(H)$, where  $H\subset X$ is a hyperplane section.

The results in Subsection \ref{hilbtris}, in  Remarks \ref{linsyz} and \ref{gentris} suggest that, under some mild assumptions,  trisecant flop contractions might exist.

We shall now construct explicitly a flop of the two dimensional irreducible components of $T$ ruled by trisecant lines to $S$ via $\phi$ as soon as  $S\subset\p^5$ has the expected trisecant behaviour and $\tilde \phi$ is a small contraction. When these loci exhaust the exceptional locus of a trisecant  flop contraction we shall obtain  a  {\it trisecant flop} of $\tilde\phi:X'\to Y$. Flops of this kind have been also considered in \cite{LLW} in arbitrary dimension under the stronger assumption that the splitting \eqref{splitT'} holds for every line of the ruling of $T'$.

\begin{thm} {\rm ({\bf Trisecant flop})}\label{trisflop} Let notation be as above, suppose that $S\subset\p^5$  satisfies Assumption~\ref{assumption1} and that it has the expected trisecant behaviour. If $T'\subset X'$ denotes the  exceptional locus of the associated
small contraction 
$$\tilde\phi:X'=\Bl_SX\to Y,$$ then  
any irreducible smooth surface $\overline T\subseteq T'$, which is ruled via $\tilde\phi$ by the strict transforms of trisecant lines to $S$ (that is through a general point of $\overline T'$ there passes a unique curve of this kind),  can be flopped to  produce a small contraction  $\tilde\psi:W'\to Y$ with $W'$ a smooth projective fourfold. 

In particular, under the previous assumptions, if $\tilde\phi:X'\to Y$ is a trisecant flop contraction and if $T'\subset X'$ is a  smooth irreducible surface ruled via $\tilde\phi$ by  trisecant  lines, then the trisecant flop  $\tau:X'\map W'$ exists.
\end{thm}
\begin{proof} First we shall first prove the second part,  that is suppose that   $\overline T=T'$ is a smooth irreducible surface
ruled via $\tilde\phi$ by trisecant lines and such that $\tilde\phi(T')=\overline C\subset Y$ is a  curve. At the end we shall consider the general case in which $T'$ is a finite union of such surfaces.  The general  fiber of $\tilde\phi:T'\to \overline C$ is smooth and irreducible so that $\overline C$ generically coincides, as a scheme, with the parameter space of trisecant lines to $S$ contained in $X$. In particular this parameter space is generically smooth of dimension one. Let $L'\subset T'$ be a general fiber of the restriction of $\tilde\phi$  to $T'$. By Lemma~\ref{split1} (see also Remark  \ref{gentris}):
\begin{equation}\label{sflop} 
N^*_{T'/X'_{|L'}}\simeq \Oo_{\p^1}(1) \oplus \Oo_{\p^1}(1).
\end{equation}

Let $\sigma:X''=\Bl_{T'}X'\to X'$ be the blow-up of $X'$ along $T'$, let $E'\subset X''$ be the exceptional divisor and let $C_1\simeq \p^1\subset E'$ be a positive dimensional fiber of $\sigma$.
By \eqref{sflop} we deduce that  $\Sigma_{L'}=\sigma\inv(L')\simeq \p^1\times \p^1$ and we can suppose that the restriction of $\sigma$ to 
$\Sigma_{L'}$ is identified with the projection onto the first factor of $\p^1\times \p^1$. Since $\Sigma_{L'}$ is also the general fiber of $\tilde\phi\circ\sigma_{|E'}:E'\to \overline C,$ we have $N_{\Sigma_{L'}/E'}\simeq\mathcal O_{\Sigma_{L'}}$. Let $C_2\simeq \p^1\subset E'$ be a fiber  of the projection onto the second factor of $\Sigma_{L'}.$  Since $N_{C_2/\Sigma_{L'}}\simeq \Oo_{\p^1}$ and since $N^*_{E'/X''_{|C_2}}\simeq\Oo_{\p^1}(1)$ by \eqref{sflop},  the exact sequence:
$$0\to N_{C_2/\Sigma_{L'}}\to N_{C_2/E'}\to N_{\Sigma_{L'}/E'_{|L'}}\simeq \Oo_{\p^1}\to 0$$
yields 
\begin{equation}\label{sc2}
N_{C_2/E'}\simeq \Oo_{\p^1} \oplus \Oo_{\p^1}
\end{equation} 
and hence $-K_{E'}\cdot C_2=2$ by Adjunction Formula. From $\sigma(C_2)=L'$ and from projection formula we get  $\sigma^*(H')\cdot C_2=1$. 
The Hilbert scheme of curves contained in the smooth projective variety  $E'$ is smooth and of dimension 2 at the point $[C_2]$ corresponding to $C_2\subset E'$ by \eqref{sc2}. So $[C_2]$ belongs to a unique irreducible component $\mathcal C$ of the Hilbert scheme with $\dim(\mathcal C)=2$. Since $\sigma^*(H')\cdot C_2=1$ and since $-K_{E'}\cdot C_2=2$, the possible deformations of $C_2$ inside $E'$ are: either irreducible and isomorphic to $\p^1$; or they consists of two distinct smooth irreducible rational curves $F_1, F_2\subset E'$ intersecting in one point and such that $N_{F_i/E'}\simeq \Oo_{\p^1}\oplus\Oo_{\p^1}(-1)$, see \cite[Sections 3.24, 3.25, 3.25.2]{Mori}. Since $1=\sigma^*(H')\cdot (F_1+F_2)$, we can suppose $\sigma^*(H')\cdot F_1=1$ (yielding $E'\cdot F_1=-1$) and $\sigma^*(H')\cdot F_2=0$. The last condition means that either $F_2$ is contracted to a point by $\sigma$ and hence $E'\cdot F_2=-1$ by \eqref{sflop} or that $\sigma(F_2)$ is a curve contracted to a point by $\lambda$, that is a positive dimensional fiber of the blow-up $\lambda$. The first case is excluded because  it would imply $E'\cdot (F_1+F_2)=-2$, contradicting  $E'\cdot C_2=-1$. If $\sigma(F_2)\simeq\p^1\subset T'$, then it is a $(-1)$ curve on $T'$. From  $(3H'-E)\cdot \sigma(F_2)=-E\cdot \sigma(F_2)=1$ we would deduce that $\overline C=\tilde\phi(\sigma(F_2))$ is a line and that the projection via $\lambda$ of all the fibers of $T'\to \overline C$
pass through the smooth point $\lambda(\sigma(F_2))=p\in T\cap S$. Then $T\subset X$ would be a plane because it would coincide with its tangent plane at $p$. The intersection  $T\cap S$ would contain a cubic curve since $S$ is scheme theoretically defined by cubics and every line contained in $T$ would be trisecant to $S$. The plane $T$ would be contracted to a point by $\phi$ (and a fortiori by $\tilde\phi$) so that $T$ would not be ruled by trisecant lines to $S$. In conclusion the deformations of $C_2$ inside $E'$ are all smooth, irreducible and isomorphic to $\p^1$, parametrized by a smooth projective surface (the splitting \eqref{sc2} necessarily holds for all the deformations of $C_2$)
and the locus of these curves is $E'$. The extremal ray $\mathbb R_+[C_2]$ determines a contraction $\omega': E'\to R'$ with
$R'$ a smooth surface and such that every fiber of $\omega'$ is isomorphic to $\p^1$, see \cite[Theorem 3.5.1]{Mori}.

By the above analysis the surface $R'$ is ruled by the curves $L''=\omega'(\Sigma_{L'})$. 
There exists a morphism $\omega:X''\to W'$ with $W'$ a smooth irreducible projective fourfold, which is the blow-up of $W'$ along the
smooth surface $R'$ with  exceptional divisor  $E'$ and whose restriction to $E'$ is $\omega'$, see for example \cite{ Nakano, FNakano} and also \cite{Artin}.

The smooth rational curves $L''\subset R'$ are disjoint and contracted to $\overline C$ by the nef, big and base point free linear system $|-K_{W'}|$, yielding  a morphism $\tilde\psi:W'\to Y$ such that $\overline C=\tilde\psi(R')$ and such that the surface $R'$ is ruled by $\tilde\psi:R'\to \overline C$.

Suppose now that $T'=T_1\cup\cdots\cup T_r$, $r\geq 2$,  with  $T_i$ a smooth irreducible projective surface  ruled via $\tilde\phi$ by trisecant lines to $S$. After applying the previous construction to $T_1$ we produce $W_1$ and we have changed $T_1$ with $R'_1$. The strict transforms $T'_j\subset W_1$ of $T_j$, $j=2,\ldots r$, are smooth irreducible surfaces which are ruled by the strict transform of trisecant lines to $S$ and such that the restriction of $N^*_{T'_j/W_1}$ to a general fiber of $\tilde\psi:T'_j\to C_j$ satisfies \eqref{sflop}. Then we can flop $T'_2$ and produce a smooth fourfold $W_2$. After $r$ steps we get a smooth fourfold $W_r$ birational to $X'$ in which the smooth irreducible ruled surfaces $T_j$ have been changed with the corresponding $R'_j$.  The birational map $X'\map W_r$ is an isomorphism in codimension 1 but not an isomorphism and it has been factorized into a sequence of elementary flops (see also \cite{HMK} for the general program of factorization of birational maps into elementary links according to Sarkisov). 
\end{proof}
We now state a useful corollary, helpful for our applications and showing that the  phenomenon described above really occurs.
\begin{cor}\label{K3flop} Let $S\subset\p^5$ be a smooth surface satisfying Assumption~\ref{assumption1}, condition $\mathcal K_3$ and having the expected trisecant behaviour.
If $X\subset\p^5$ is a cubic hypersurface through $S$ not containing any plane spanned by cubic curves on $S$ and if $T'\subset X'$ denotes the  exceptional locus of the associated
small contraction 
$\tilde\phi:X'=\Bl_SX\to Y,$
then there exists the trisecant flop $\tilde\psi:W'\to Y$ of any irreducible component of $T'$.
\end{cor}

\begin{Remark}\label{hypothesis}  Obviously,  one might only assume that $|-K_{X'}|=|3H'-E|$ is generated by global sections and big (or only nef and big but not generated by global sections) without requiring that the trisecant flop contraction is necessarily given by $|-K_{X'}|$. It is not difficult to see that in any case, for some $m\geq 1$, the linear system  $|-mK_{X'}|$ gives a trisecant flop contraction. We avoided this more general approach to simplify the exposition but  there are examples of trisecant flops appearing also  in more general settings, see  Subsection 3.6 of the first arXiv version of this paper and also example (xv)  of Table~\ref{tabella}.
\end{Remark}

\subsection{Trisecant flop and congruences of \texorpdfstring{$(3e-1)$}{(3e-1)}-secant rational curves of degree \texorpdfstring{$e\geq 2$}{e>=2}} 
The aim of this section is to relate the trisecant flop to (congruences of) $(3e-1)$--secant curves to $S$. We start by an easy but
very useful result.

\begin{prop} {\rm ({\bf Extremal ray generated by $(3e-1)$-secant curves})}\label{extremalray} Let notation  be as above. Suppose that  $S\subset\p^5$
satisfies Assumption~\ref{assumption1} and that there exists the trisecant flop $\tilde\psi:W'\to Y$ of a  trisecant flop contraction $\tilde\phi:X'=\Bl_SX\to Y$
with $X$ a general cubic fourfold through $S$. 

If $C'\subset X'$ is the strict transform on $X'$ of a $(3e-1)$-secant curve to $S$ of degree $e$ contained in $X$, then the strict transform $\overline C'$ of $C'$ on $W'$ generates an extremal ray on $W'$.
\end{prop}
\begin{proof} By hypothesis there exists a trisecant flop of the trisecant flop contraction $\tilde\phi:X'=\Bl_SX\to Y$ and hence a commutative diagram:
\begin{equation}\label{diagramma2}
\UseTips
 \newdir{ >}{!/-5pt/\dir{>}}
 \xymatrix{
 &\Bl_TX'=\Bl_RW'\ar[ld]_\sigma\ar[rd]^{\omega}&\\            
 X' \ar@{-->}[rr]_{\tau}\ar[rd]_{\tilde\phi}\ar[d]_\lambda&                             &W'\ar[ld]^{\tilde\psi}\\
X &Y&.
}
\end{equation}
By definition $C'$ is the strict transform of a  curve of degree $e$, which is $(3e-1)$--secant to $S$. Thus $$K_{X'}\cdot C'=(E-3H')\cdot C'=3e-1-3e=-1.$$ Consider the possible degenerations $C''$ of $C'\subset X'$ as sum of effective 1--cycles inside $X'$:
$$C''= C_1+C_2.$$
The cycles $C_1$ and $C_2$ have  degree $e_i=H'\cdot C_i$ and  are $\beta_i=E\cdot C_i$ secant to $S$. In particular $e_1+e_2=e$ and $\beta_1+\beta_2=3e-1$. 
Since $-K_{X'}$ is nef and since
$$1=-K_{X'}\cdot C'=-K_{X'}\cdot (C_1+C_2),$$
 either $K_{X'}\cdot C_1=0$ or $K_{X'}\cdot C_2=0$. In conclusion either $[C_2]=[e_2L']$
or $[C_1]=[e_1L']$ with $L'$ the strict transform of a trisecant line to $S$ contained in $X$. Let $\overline C'\subset W'$ be the strict transform of $C'\subset X'$. Then $[\overline C']$ generates an extremal ray because $\tilde\phi$ has contracted all the rational curves in $\mathbb R_+[L']$.  
\end{proof}

The locus of the extremal ray $\mathbb R_+[\overline C']$ will determine the type of the associated elementary Mori contraction from $W'$ onto a $\mathbb Q$-factorial Fano variety. Here we shall  consider only  the most relevant case for our applications. Examples of fiber type contractions can be constructed as soon as $|(3e-1)H'-eE|\neq\emptyset$ and $S\subset\p^5$ has a finite number $\rho\geq 2$
of $(3e-1)$-secant curves of degree $e\geq 2$ passing through a general point of $\p^5$. The dimension of the general fiber of the contraction will be $\rho-1=4-\dim(\mu(X))$, where $\mu$ is the rational map on $X$ defined by the linear system of hypersurfaces of degree $3e-1$ having points of multiplicity at least $e$ on $S$.

\begin{thm} {\rm ({\bf Extremal contraction of the congruence})}\label{contraction}  Let notation  be as above. Suppose that  $S\subset\p^5$
satisfies Assumption~\ref{assumption1}, that it has the expected trisecant behaviour and that there exists the trisecant flop $\tilde\psi:W'\to Y$ of the trisecant flop contraction $\tilde\phi:X'=\Bl_SX\to Y$
with $X$ a general cubic fourfold through $S$ and with $T'$ irreducible. 

If $S\subset\p^5$ admits a congruence $\pi:\mathcal D\to \mathcal H$ of $(3e-1)$--secant rational curves of degree $e\geq 2$, then  the locus of curves
of the congruence contained in  $X\subset\p^5$ is an irreducible divisor $D\subset X$ and the following hold:
\begin{enumerate}

\item  there exists a divisorial contraction $\nu:W'\to W$, with
$W$ a locally $\mathbb Q$--factorial projective Fano  variety, whose exceptional locus $\overline E$ is the strict transform of $D$ on $W'$ and such that $\nu(D)=U$ is an irreducible surface  supporting  the base locus scheme $B$ of $\nu\inv$.  The base locus scheme $B$ is generically smooth , irreducible and $\nu$ is generically the blow--up of the surface $U$. 

\item  
The induced birational map $\mu':X'\map W$ (or $\mu:X\map W$) is given by a linear system in $|(3e-1)H'-eE|$.

\item Let $\overline H'=\nu^*(\overline H)$ with
$\overline H\subset W$ a generator of $\Pic(W)$ and let $-K_W=i(W)\overline H$.
The induced birational morphism $\tilde\psi:W'\to Y$ is given by a linear system in $|i(W)\overline H'-\overline E|$ while the birational map
$W'\map X$ is given by a linear system in $|(i(W)\cdot e-1)\overline H'-e\overline E|$. The strict transform $D'\subset W$ of $E$ via $\mu'$ 
is a locus of  $(i(W)\cdot e-1)$--secant curves to $U$ of degree
$e$  such that through a general point of $D'$ there passes a unique curve of the family. 

\item The irreducible components of $T$ are contained in the base locus scheme of $\mu$ and their flopped images on $W$ are contained in the base locus scheme of $\mu\inv$.
The flopped images of the scrolls in $T$ are scrolls $R\subset W$ ruled by  {\it lines} in $W$, which are $i(W)$-secant to $U$.
\end{enumerate}
\end{thm}
\begin{proof} Let notation be as above and let $p:\mathcal D\to \p^5$ be the tautological morphism of the congruence and let $\mathcal H$ be its parameter space. Since $p$ is birational, the locus $\mathcal E\subset\p^5$ of points through which there passes more than one curve of the congruence has codimension at least two in $\p^5$ by Zariski Main Theorem. Since the curves of the congruence are $(3e-1)$-secant to $S$ and hence to $X$, to be contained in $X$ imposes two conditions in $\mathcal H$ by Bézout Theorem. Putting these two facts together we deduce that $D$ is a divisor inside $X$, whose irreducibility will be proved below. By hypothesis there exists a trisecant flop of the trisecant flop contraction $\tilde\phi:X'=\Bl_SX\to Y$ and hence the commutative diagram \eqref{diagramma2}.
Let $D'\subset X'$ be the strict transform of $D$ on $X'$ and let $C'\subset D'$ be the strict transform of a general curve $C$ of the congruence $\pi:\mathcal D\to \mathcal H$  contained in $X$. By definition $C'$ is the strict transform of a smooth rational curve of degree $e\geq 2$ which is $(3e-1)$--secant to $S$. 

Let $\overline C'\subset W'$ be the strict transform of a general curve of the congruence $C'\subset D'$ and let $\overline E\subset W'$ be the strict transform of $D$. Then $K_{W'}\cdot \overline C'=-1$ and $[\overline C']$ generates an extremal ray by Proposition \ref{extremalray}.
 By construction the locus of the extremal ray $\mathbb R_+[\overline C']$ is the divisor $\overline E$ and $N_{\overline C'/W'}\simeq \Oo_{\p^1}\oplus\Oo_{\p^1}\oplus\Oo_{\p^1}(-1)$, see for example \cite[Lemma 2.5]{Ando}. 
By \cite[Theorem 2.1]{Ando} there exists a birational divisorial contraction $\nu:W'\to W$ of the extremal ray $\mathbb R_+[\overline C']$ with $W$ a locally $\mathbb Q$-factorial projective variety of dimension four with $\Pic(W)=\mathbb Z\langle\overline H\rangle$ and with $\overline H$ ample. If $i(W)$ is defined by $-K_W=i(W)\overline H$, then  $i(W)>0$ since $W$ is birational to $X$. Moreover, the divisor $\overline E$ (and hence the divisor $D$) is irreducible being the exceptional locus of $\nu$ (recall that $\Pic(W')\simeq \mathbb Z\oplus\mathbb Z$).

Let $U=\nu(\overline E)\subset W$. Since the $(3e-1)$-secant curves to $S$ belong to a congruence, through a general point of  $D$ there passes a unique curve of the congruence, as recalled at the beginning of the proof. So  the same holds for $\overline E$ and the restriction of $\nu$ to $\overline E$ has general fiber isomorphic to a curve $\overline C'$, giving $\dim(U)=2$. In particular $N_{\overline C'/\overline{E}}\simeq 
\Oo_{\p^1}\oplus\Oo_{\p^1}$ for a general $\overline C'$ so that  the previous splitting of the normal bundle $N_{\overline C'/W'}$ yields  $\overline E\cdot \overline C'=-1$. Therefore, there exists an open subset  $U_0\subset U$ consisting of smooth points of $U$ and such that the fiber of $\nu:\overline E\to U$ over a point of $U_0$ is isomorphic to $\p^1$. Then $\overline E_0=\nu^{-1}(U_0)\subseteq \overline E$ is smooth and, letting  $W_0\subset W$ be an open subset such that $W_0\cap U=U_0$,  $\nu\inv(W_0)\to W_0$ is the blow-up of $W_0$ along $U_0$,  see for example \cite{Nakano, FNakano}. In particular the base locus scheme of $\nu\inv$, which is supported on $U$, coincides generically with $U$ and hence it is generically smooth. All the assertions in (1) are now proved.

The birational map $\mu'=\nu\circ\tau:X'\map W$ is given by a linear system in $|a[(3e-1)H'-eE]|$, $a\geq 1$.  Indeed, $\Pic(X')\simeq \mathbb Z\langle H'\rangle\oplus\mathbb Z \langle E\rangle$ and such a divisors is of the forme $\alpha H'-\beta E$ with $\alpha>0$ and $\beta>0$.  From 
$$0=(\alpha H'-\beta E)\cdot C'=\alpha e-\beta(3e-1)$$
and from  $e\geq 2$, we get $\alpha=a(3e-1)$ and $\beta=a e$ with $a\geq 1$.
The irreducible components
of $T$ are contained in the base locus scheme of this linear system  because 
$$a[(3e-1)H'-eE]\cdot L'=-a<0$$
for every strict transform of a general $3$-secant line $L$ to $S$ contained in $X$. Since the map $\mu'$ is compatible
with the trisecant flop, necessarily   $a=1$. Indeed, after the blow-up of $T$ the birational map $\mu'$ becomes a morphism so that 
\begin{equation}\label{flopcomp}
[N^*_{T'/X'}\otimes \O_{X'}(a[(3e-1)H'-eE])]_{|L'}\simeq  \Oo_{\p^1}(1-a) \oplus \Oo_{\p^1}(1-a)
\end{equation}
 is generated by global sections, yielding $a\leq 1$ and hence $a=1$, concluding the proof of (2).

 We have a commutative diagram:
\begin{equation}\label{diagmu}
\UseTips
 \newdir{ >}{!/-5pt/\dir{>}}
 \xymatrix{
 &\Bl_{T'}X'=\Bl_{R'}W'\ar[ld]_\sigma\ar[rd]^{\omega}&\\             
 X' \ar@{-->}[rr]_{\tau}\ar[d]_\lambda\ar@{-->}[rrd]_{\mu'} &                             &W'\ar[d]^\nu \\
X\ar@{-->}[rr]_{\mu} & & W
}
\end{equation}

 Let $\overline H\subset W$  be as above, let
  $\overline H'$ be its strict transform on $W'$ and, keeping notation as in the proof of Theorem \ref{trisflop}, let $L''=\omega(\Sigma_{L'})\subset W'$ be a general fiber of the ruling of the smooth surface $R'=\omega(E')\subset W$. Since $\nu:W'\to W$ is generically the blow-up of $U$, we have  $-K_{W'}=i(W) \overline H'-\overline E$.  The morphism $\nu\circ\omega: \Bl_RW'\to W$ is given by a linear system in $|(3e-1)\sigma^*(H')-e\tilde E-E'|$ with $\tilde E$ the strict transform of $E$ on $\Bl_TX'$. Then 
 $$\nu(L'')\cdot \overline H=F_1\cdot ((3e-1)\sigma^*(H')-e\tilde E-E')=-(F_1\cdot E')=1,$$
  that is $\nu(L'')\subset W$ is a {\it line} with respect to $\overline H$. By Projection Formula we deduce $L''\cdot \overline H'=1$
  and $0=-K_{W'}\cdot L''=i(W)-\overline E\cdot F'$  yields 
  $i(W)=\overline E\cdot L''$, that is the {\it lines}  $\nu(L'')$ are $i(W)$-secant to $U\subset W$.  We also have $\overline C'\cdot \overline E=1$ and  $\overline H'\cdot \overline C'=0$. Since the birational morphism $\tilde\psi:W'\to Y$
  is given by a linear system in $|i(W)\overline H'-\overline E|$ as shown above, the birational map $\psi=\tilde\psi\circ\nu\inv:W\map Y$ is given by a linear system of divisors in $|O_W(i(W))|$ vanishing on $U$.

We have $\Pic(W')\simeq \mathbb Z\langle \overline H'\rangle\oplus\mathbb Z\langle \overline E\rangle$ so that the map $\eta'=\mu\inv\circ\nu:W'\map X$ is given by a linear system in $|\alpha'\overline H'-\beta'\overline E|$ with $\alpha'>0$ and $\beta'>0$. Since  general fiber $\overline C'$ of $\nu:\overline E\to U$ is sent into a curve of the congruence $\mathcal D$, which by definition  has degree $e\geq 2$,
we deduce
$$e=(\alpha \overline H'-\beta\overline E)\cdot \overline C'=\beta.$$
Moreover, reasoning as above, the compatibility with the trisecant flop yields
$$-1=(\alpha\overline H'-e\overline E)\cdot L''=\alpha-e\cdot i(W),$$
that is  $\alpha=i(W)\cdot e-1.$
  In conclusion, the birational map $\eta'$ 
 is given by a linear system in $|(i(W)\cdot e-1)\overline H'-e\overline E|$ of dimension 5 and $\mu\inv$ is given by a linear system of dimension 5 of divisors
 in $|H^0(\Oo_W((i(W)\cdot e-1))|$ having points of multiplicity at least $e$ along $U\subset W$.  The previous analysis shows that the base locus of $\mu\inv$ contains $U$ and $\nu(R)$, which is a locus of $i(W)$-secant {\it lines} to $U$ contained in $W$, concluding the proof of (4). 
 
 Let $D'\subset W$ be the strict transform of $E\subset X'$ via $\mu$, which is an irreducible divisor. Let $F\subset E$ be a positive dimensional fiber of $\lambda:X'\to X$. Then $D'$ is ruled by the strict transforms of the curves $F$ and such a general point of $D'$ there passes a unique curve of this family. Moreover,
$F\cdot [(3e-1)H'-eE]=e$ and, letting $\overline F=\tau(F)\subset W$ and recalling that $\eta'(\overline F)$ is a point, we have 
$$0=\overline F\cdot  [(i(W)\cdot e-1)\overline H'-e\overline E]=e[(i(W)\cdot e-1)-\overline F\cdot \overline E],$$ 
yielding $\overline F\cdot\overline E=i(W)\cdot e-1$. In conclusion, $\mu'(F)\subset W$ is a curve of degree $e$ with respect to $\overline H$ which is $(i(W)\cdot e-1)$-secant to $U\subset W$. This proves the last assertion in (3).
\end{proof}

\begin{Remark}\label{Mukai} Obviously,  one can also reverse the construction in Theorem \ref{contraction} starting from suitable $U\subset W$ and then producing the congruences of $(3e-1)$-secant curves of degree $e$ to a surface $S\subset X\subset\p^5$ by taking the image of $\overline E$ in $X$ and by taking $S\subset X\subset\p^5$ as the surface describing the linear system defining the inverse map $\mu:X\map W$. In practice, as soon as the trisecant flop exists, the existence of a congruence of $(3e-1)$-secant lines to $S$ is equivalent to the existence  of the surface $U\subset W$, which should be an incarnation of the associated K3 surface (see the end of subsection \ref{s38} for the analysis of the case  $d=38$ to see one such explicit incarnation). From this point of view one associates to the pair $(X,S)$ a pair $(W,U)$, where $W$ is the image of $X$ and  the surface $U$ is naturally the parameter space of the curves of the congruence $\mathcal D$ contained in $X$.
\end{Remark}

\section{Associated K3 surfaces to cubic fourfolds in $\mathcal C_{38}$ via the trisecant flop}
\label{explicitflops} 

In this section, as an application of previous theoretical results, we describe
birational incarnations of the  K3 surfaces {\it associated} to the cubic fourfolds in $\mathcal C_{38}$ 
via Hodge Theory or via Derived Category Theory. For the sake of brevity, we omit a similar analysis of the cases 
of cubic fourfolds in
$\mathcal C_{14}$ and $\mathcal{C}_{26}$, which has been outlined  in the first arXiv version of this paper.
The case of cubic fourfolds in $\mathcal C_{42}$ will be considered in Subsection~\ref{TF42}.

\subsection{\bf General properties of  degree 10 smooth surfaces \texorpdfstring{$S_{38}\subset\p^5$}{S38 in P5} of sectional genus 6}\label{s38}
Let us consider the smooth surfaces $S_{38}\subset\p^5$   obtained as the image
of $\p^2$ by the linear system of plane curves of degree 10 having 10 fixed triple points in general position.
These surfaces are contained in a general cubic fourfold in the admissible divisor $\mathcal C_{38}$, as shown in \cite{Nuer}. They  were also studied in \cite{RS1, Explicit} to prove that every cubic fourfold in $\mathcal C_{38}$ is rational. 

The Hilbert scheme $\mathcal S_{38}$ parametrizing such surfaces is explicitly unirational,  that is
we can write out equations for the general member 
$[S_{38}]\in\mathcal S_{38}$ over a pure transcendental extension of the base field.
From this, one can deduce that 
the homogeneous ideal of $S_{38}$ is generated by 10 cubic forms, whose first syzygies are generated by 
the linear ones.
In particular the general $S_{38}\subset\p^5$  satisfies  condition $\mathcal K_3$. By Proposition \ref{verme} the linear system 
$|H^0(\mathcal I_{S_{38}}(3))|$ defines a birational map $\phi:\p^5\map Z\subset\p^9$ onto its image $Z$. Through a general point $\phi(p)\in Z$ there passes 8 lines contained in $Z$. The pullbacks of these lines are seven secant lines to $S_{38}$ passing through $p$ and a 5-secant conic to $S_{38}$. In particular, a general  $S_{38}\subset\p^5$ admits a congruence of 5-secant conics  (see \cite[Section~5]{RS1} for details on this computation and also Subsection \ref{contid38}).

Moreover, we have 
$|H^0(\mathcal I_{S_{38}}^2(5))|=\PP^4$ for a general $S_{38}\subset\p^5$  and the coefficients of the multidegree of the graph of the associated rational map $\mu:\p^5\map\p^4$ are $(1, 5, 19, 13, 2)$ (see Subsection \ref{contid38}).  From this one deduces that: $\mu$ is dominant since the last entry  is equal to 2 (this means that the closure of a general fiber of $\mu$, $F=\overline{\mu\inv(\mu(p))}$ with $p\in\p^5$ general has degree 2 and dimension 1); its base locus scheme $B\subset\p^5$  has  degree $6=5^2-19$ and dimension three. Since the unique 5-secant conic $C_p$ to $S_{38}$ passing through $p$ is contracted  by $\mu$ to the point $\mu(p)$ ($5\cdot 2-2\cdot 5=0$), the fiber $F$ coincides with $C_p$ and it is  irreducible (otherwise one can argue as in Subsection \ref{contid38}, prove that $F$ is an irreducible conic and verify that it is 5-secant to $S_{38}$, yielding a different proof of the existence of the congruence of 5-secant conics).
 
The rationality of a general $X$ through a general $S_{38}$ follows by restricting $\mu$ to $X$. Indeed, through a general $q\in X$ there passes a unique conic $C_q$ of the congruence, which is not contained in $X$ by the generality of $q$. The conic $C_q$ cuts $X$ in $q$ and in five points on $S_{38}$, yielding that  the general fiber of  map $\mu=\mu_{|X}:X\map\p^4$ is a point by Bézout Theorem.

\subsection{Small contraction defined by cubics through a general $S_{38}\subset\p^5$}\label{small38}

The (closure of the) fibers of $\phi:\p^5\map Z$ are linear spaces of dimension $s$ with $0\leq s\leq 2$. 
The two  dimensional fibers of $$\tilde\phi:\Bl_{S_{38}}\p^5 \to Z\subset\p^9$$ 
are  the strict transforms of planes in $\p^5$ cutting  $S_{38}$ along plane cubic curves by Proposition \ref{verme}. Let  $C\subset S_{38}\simeq \Bl_{\{p_1,\ldots, p_{10}\}}\p^2\subset\p^5$ be such a cubic and recall that  the embedding is given by $|10 H-\sum_{i=1}^{10}3E_i|$, using  the standard notation. Since the curve  $C$ is contained in a plane, it  cuts each line $E_i\subset S_{38}$ in at most one point so $C\equiv \alpha H-\sum_{i=1}^{10}a_iE_i$ with $0\leq a_i\leq 1$ and $\alpha\geq 1$. From $$3=C\cdot (10H-\sum_{i=1}^{10}3E_i)=10\alpha-3\sum_{i=1}^{10}a_i,$$ 
we deduce $\alpha=3$ and  $a_i=1$ for nine of the ten indices. Hence $C$ is  the image of a plane cubic curve passing through  9 of the 10 general base points on $\p^2$. In conclusion there are 10 two dimensional fibers of $\tilde\phi$. The other positive dimensional fibers  of $\tilde\phi$ are 
the strict transforms of trisecant lines to $S_{38}$. From the equations of the base locus scheme $B$ of $\mu$, we deduce that $B$ has a unique irreducible component $\tilde B$ of dimension three and degree 6, which contains $S_{38}$. In particular, $B$ is generically reduced along $\tilde B$.   An explicit computation shows that the variety $\tilde B$ is mapped by  $\phi$ onto an irreducible surface  $V\subset\ Z\subset\p^9$, which is  a  Veronese surface generating a $\p^5\subset\p^9$. The general fiber of the restriction of $\phi$ to $\tilde B$ has  dimension one and hence it is a trisecant line to $S_{38}$.  A trisecant line to $S_{38}$ is contained in the base locus scheme $B$ of $\mu$ because it intersects a quintic with double points along $S_{38}$ in at least six points counted with multiplicity. Hence  $\tilde B$ is the unique irreducible component of $\Trisec(S_{38})$ of dimension 3 and $S_{38}$ has the expected trisecant
behaviour because the irreducible component of $\Al^3 S_{38}$ corresponding to $\tilde B$ is birational to the smooth irreducible surface $V$. From this analysis and from the previous computations, we deduce that $\Trisec(S_{38})$ consists of $\tilde B$ and of the 10 planes cutting $S_{38}$ along cubic curves.

Since the ten planes in $\Trisec(S_{38})$ are mapped to ten points in $Z$, a general hyperplane section of $Z$ does not pass through these 10 points. Hence a general cubic hypersurface $X\subset \p^5$ through $S_{38}$ does not contain any of the 10  planes in $\Trisec(S_{38})$. The restriction of $\tilde\phi$ to $X'=\Bl_{S_{38}}X$ induces a small contraction:
$$\tilde\phi:X'=\Bl_{S_{38}}X\to Y\subset\p^8,$$
with $Y=Z\cap H$ the corresponding hyperplane section of $Z$. From the previous description we deduce that  $\tilde B\cap X=T\cup S_{38}$ 
with  $T\subset X$  an irreducible surface by the generality of $X$. Moreover $\deg(T)=3\times 6-10=8$ and 
$\phi(T)=C=V\cap H\subset Y=Z\cap H\subset \p^8$ is a smooth rational normal quartic curve (a  general  hyperplane section of the Veronese surface $V\subset Z$). Hence $T$ is ruled by the trisecant lines to $S_{38}$ contained in $X$ via the restriction of $\phi$. The Double Point Formula yields that the  singular locus of the rational scroll $T$, projection of a smooth rational normal scroll of degree 8 in $\p^9$, consists of six singular points.  Its strict transform  $T'\subset X'$ is smooth and $\tilde\phi|_{T'}:T'\to C$ is a $\p^1$-bundle. The birational morphism $\tilde\phi$ is an isomorphism between $X'\setminus T'$ and $Y\setminus  C$ and hence it is a trisecant flop contraction.

\subsection{The trisecant flop determined by $S_{38}\subset\p^5$}\label{desTF38}

Theorem \ref{trisflop} and Theorem \ref{contraction} 
assure the existence of the trisecant flop  $\tilde\psi:W'\to Y$
and of the divisorial contraction $\nu:W'\to \p^4$ giving a factorization of the birational map $\mu:X\map \p^4$ (see the commutative diagram \eqref{diagmu} also for recalling the notation).

The scroll  $\overline R=\nu(R')\subset\p^4$ has degree 6 (recall the we tensor with $\Oo_{\p^1}(-1)$ performing the flop, see \eqref{flopcomp}) and $\mu^{-1}:\p^4\map X$ is given by a linear system in $|H^0(\mathcal I_U^2(9))|$ by part (3) of Thereom \ref{contraction}, where $U\subset\p^4$ is the support of the base locus $\overline B$ of $\nu\inv$ (recall that $\overline B$ is generically reduced so that it coincides with $U$ generically). By part (4) of Theorem \ref{contraction}, the lines of the scroll $\overline R$ are 5-secant to $U$, the map $\psi$ is given by a linear system in $|H^0(\mathcal I_U(5))|$ and $\psi(\overline R)=\tilde\psi(R')= C$. 

The cubics through $S_{38}$ restricted to $X$ are mapped by $\mu$ onto quintics  
 defining 
$\overline B$ as a scheme by part (3) of Theorem \ref{contraction}. Taking  a basis of cubics $X_i$ through $S_{38}$ restricted to $X$, $i=1,\ldots, 9$, their images $V_i=\mu(X_i)\subset\p^4$ determine the ideal of $\overline B$. In Subsection \ref{contid38}  and in the ancillary file 
\verb!code_section_6.m2! we verified that  $\overline B$ is a smooth surface of degree 12 and sectional genus 14 so that it coincides with  $U$ as scheme. Hence $U\subset\p^4$ is a smooth surface of degree 12 and sectional genus 14, whose ideal is generated  by 9 forms of degree five. Since $\nu\inv(\overline B)=\overline E$, the universal property of 
the blow-up yields a birational projective morphism $\delta:W'\to \Bl_U\p^4$ between smooth projective fourfolds. Since $\rk(\Pic(W'))=2=\rk(\Pic(\Bl_U\p^4))$,  we have $W'\simeq\Bl_U\p^4$ by Zariski Main Theorem. Moreover,  $$\mu\inv:\p^4\map X\subset\p^9$$ is given by the linear system 
$|H^0(\mathcal I_{U}^2(9))|$ by part (3) of Theorem \ref{contraction}.

Let $A=\mathbb C[t_0,\ldots,t_4]$. Using  the explicit equations constructed above,  we can compute the resolution of the homogenous ideal of $U$:
$$0\leftarrow I_U\leftarrow A(-5)^{\oplus 9}\leftarrow A(-6)^{\oplus 11}\leftarrow A(-7)^{\oplus 3}\leftarrow 0$$
and also verify that  $p_g(U)=1$ and $q(U)=0$, see Subsection  \ref{contid38}. To conclude the analysis of $U\subset\p^4$ we shall follow the arguments in \cite[Section 2.7]{DES},
where the authors gave a different construction of the above surface via Beilinson Spectral Sequence methods.

The intersection matrix for the sublattice $\langle H, K_U\rangle$ of $\Num(U)$ is
$$\left(\begin{matrix}
H^2&H\cdot K_U\\
H\cdot K_U&K_U^2
\end{matrix}\right )=
\left(\begin{matrix}
12&14\\
14&-11
\end{matrix}\right),
$$
where $K_U^2=-11$ is deduced from the Double Point Formula for a smooth surface in $\p^4$.
The number $N_6$ of proper 6-secant lines to $U$ (if finite) plus the number of exceptional lines, that is $(-1)$-curves $L_i\subset U$ such that $H\cdot L_i=1$ and $L_i^2=-1$, is equal to 10 by Le Barz Formulas, see {\it loc. cit.}. There are no proper 6-secant lines  since $I_U$  is generated by quintic forms.
Let $U'\subset\p^{14}$ be the first adjunction surface of $U$, that is the image of the birational morphism $\gamma:U\to U'$ defined by the base point free linear system $|K_U+H|$ (see the Introduction of \cite{DES} for a summary of the basic results of adjunction theory on surfaces). Since the above morphism contracts the exceptional lines on $U$, the morphism $\gamma$ realizes  $U$ as  the blow-up of $U'$ in 10 distinct points and $K_{U'}^2=K_U^2+10=-1$.  Since $H=\gamma^*(H')-\sum_{i=1}^{10}L_i$, we also have  $H'\cdot K_{U'}=H\cdot K_U-10=4$. 

Since  $p_g(U')=p_g(U)=1$ and since $K_{U'}^2=-1$, the canonical divisor is effective but not nef and the surface $U'$ is non minimal. Let  $E\subset U'$ be a $(-1)$ curve, let $K_{U'}=D+E$ ($D\geq 0$) and let $\tilde E\subset U$ be an irreducible curve such that $\gamma(\tilde E)=E$. Then $\tilde E^2\leq E^2=-1$
and 
$$0<(K_U+H)\cdot \tilde E=(K_{U'}+H')\cdot E=H'\cdot E-1$$ yields $H'\cdot E\geq 2$. From $4=H'\cdot (D+E)$ we deduce $H'\cdot D\leq 2$ and that $E$ is not contained in the support of $D$ (otherwise $K_{U'}=2E$ and $K_{U'}^2=-4$). Then $D\cdot E=0$, $D^2=0$ and $K_{U'}\cdot D=0$. If $H'\cdot E=2$, then $H'\cdot D=2$. From $(K_{U'}+D)\cdot D=0$, we deduce that  $D$ cannot be an irreducible conic. If $D=2\cdot \tilde L$, then $H'\cdot \tilde L=1$ would be a line such $(K_{U'}+\tilde L)\cdot \tilde L=0$, which is impossible.  In the same way we can exclude that $H'\cdot E=3$, conclude that  $H'\cdot E=4$ and finally that $K_{U'}=E$. Let $\epsilon:U'\to U''$ be the contraction of $E$ to a point of the smooth surface $U''$. Since $K_{U''}=0$ and since $q(U'')=q(U)=0$, the surface  $U''$ is  a K3. The curve  $\tilde E\subset U$ is a smooth rational normal curve of degree 4, proving that the distinct points $p_i=\gamma(L_i)$ do not belong to $E$. In the next section we shall develop  an alternative geometric method to obtain the whole configuration of the $(-1)$-curves on $U$ in explicit examples.

Let  $\pi:U\to U''$ be the blow-up of $U''$ at the eleven distinct points described above and let $H''=\pi_*(H)$.
Then $H''$ is an ample divisor on $U''$ with a point of multiplicity four at $p_{11}$ and passing simply through $p_1,\ldots, p_{10}$ and
the linear system $H$ on $U$ is  thus given by $|\pi^*(H'')-\sum_{i=1}^{10}L_i-4E|$. Hence 
$(H'')^2=(H')^2+10+16=12+10+16=38$ and  $p_a(H'')=p_a(H')+6=20$. The generality of $X$ through $S_{38}$ assures that the K3 surface has general moduli and that $\Pic(U'')\simeq \mathbb Z$ (a direct proof of this fact will be given below). Hence the divisor $H''$ is very ample and gives and embedding $U''\subset \p^{20}$.

\subsection{The associated K3 surface to a general cubic in $\mathcal C_{38}$ via the trisecant flop}\label{assK3C38}

Via the trisecant flop and via the contraction of the curves of the congruence contained in $X$, we proved that the associated surface $U\subset\p^4$ to a general pair $(X,S_{38})$ is  a  birational incarnation of a  general smooth K3 surface $U''\subset \p^{20}$ of degree 38 and genus 20. We now want to  show that the surface $U''$ (or $U$) is associated to $X$ in the sense of  Hodge Theory (or, equivalently, of  Derived Category Theory) following the treatment by Hassett in  \cite[Section 3]{Levico}.
\medskip

Let us recall that for an arbitrary cubic fourfold $X\subset\p^5$, letting $h^{p,q}$ denote the Hodge numbers of $X$, we have $h^{0,4}=h^{4,0}=0$, $h^{1,3}=h^{3,1}=1$, $h^{2,2}=21$. 
Let $h$ denote the class of a hyperplane section of $X$ and let $$H^4_{\prim}(X,\mathbb Z)\simeq \langle h^2\rangle^\perp\subset H^4(X,\mathbb Z)$$ be the primitive cohomology of $X$. This cohomology reminds the $H^2$ cohomology of a K3 surface $S$, modulo a Tate twist by $-1$, since $S$ has Hodge numbers $h^{1,0}=h^{0,1}=1$, $h^{1,1}=20$.
The intersection forms have signatures: $(20,2)$ for $H^4_{\prim}(X,\mathbb Z)$ by the Hodge-Riemann bilinear relations and $(19,3)$ for $H^2(S,\mathbb Z)(-1)$. They become compatible as soon as one can find a common codimension one rank sublattice with signature $(19,2)$.

The definition of  $\mathcal C_{38}$ (see \cite[Section 2.3]{Levico} for the general theory) and the fact that a general $[X]\in\mathcal C_{38}$ contains a surface $S_{38}\subset\p^5$ (see \cite{Nuer}) yield
$$H^{2,2}(X,\mathbb{Z})=H^4(X,\mathbb Z)\cap H^1(\Omega^1_X)\simeq\mathbb{Z}\langle h^2, [S_{38}] \rangle$$ for such a $X$.
  Let $\mathcal T_X\subset H^4(X,\mathbb{Z})$ denote the transcendental part
of the cohomology of a cubic fourfold $X\subset\p^5$. Then in our setting $$\mathcal T_X=\langle h^2, [S_{38}] \rangle^\perp \subset H^4(X,\mathbb Z)$$ has rank 21 and signature $(19,2)$.
Clearly $H^2(S_{38},\mathbb Z)\simeq \mathbb Z^{11}$,  $H^2(T',\mathbb Z)\simeq \mathbb Z^2\simeq H^2(R',\mathbb Z)$ and $H^2(U,\mathbb Z)\simeq H^2(U'',\mathbb Z)\oplus \mathbb Z^{11}$ because we blow-up eleven points on $U''.$
 
Let $M$ be a smooth projective fourfold and let $S\subset M$ be a smooth projective surface.  Then 
\begin{equation}\label{h4split}
H^4(\Bl_SM,\mathbb Z)\simeq H^4(M,\mathbb Z)\oplus_{\perp}H^2(S,\mathbb Z)(-1).
\end{equation}
 The homomorphism $H^4(M,\mathbb Z)\to H^4(\Bl_SM,\mathbb Z)$ is induced by the pull-back of $\pi:\Bl_SM\to M$. Letting 
 $E=\p(N^*_{S/M})\to S$ (Grothendieck notation), the homomorphism $H^2(S,\mathbb Z)(-1)\to H^4(\Bl_SM,\mathbb Z)$ is the 
 pull-back from $S$ to $E\subset \Bl_SM$ followed by push-forward via the inclusion of $E$.

We have a commutative diagram:
\begin{equation}\label{blU}
\UseTips
 \newdir{ >}{!/-5pt/\dir{>}}
 \xymatrix{
 &\Bl_{T'}X'=\Bl_{R'}W'\ar[ld]_\sigma\ar[rd]^{\omega}&\\             
 X'=\Bl_{S_{38}}X \ar@{-->}[rr]_{\tau}\ar[d]_\lambda &                             &W'=\Bl_{U}\p^4\ar[d]^\nu \\
X\ar@{-->}[rr]_{\mu} & & \p^4
}
\end{equation}
inducing an isomorphism between 
$$[H^4(X,\mathbb Z)\oplus_{\perp}H^2(S_{38},\mathbb Z)(-1)]\oplus_\perp H^2(T',\mathbb Z)(-1)$$
and 
$$ [H^4(\p^4,\mathbb Z)\oplus_{\perp}H^2(U,\mathbb Z)(-1))]\oplus_\perp H^2(R',\mathbb Z)(-1)$$
respecting the Hodge structures.

For a smooth projective surface  $S$ let  $\mathcal T_S\subset H^2(S,\mathbb Z)$ denote  the trascendental part of the cohomology of $S$. 
Then  $\mathcal T_{R'}=\mathcal T_{T'}=\mathcal T_{S_{38}}=0$ while  $\mathcal T_{U}\simeq \mathcal T_{U''}$ and  $\rk(\mathcal T_{U''})=22-\rk(\Pic(U''))$. From one hand 
$$\mathcal T_{\Bl_{T'}X'}\simeq \mathcal T_{\Bl_{S_{38}}X}\simeq \mathcal T_X=\langle h^2,[S_{38}]\rangle^\perp \subset H^4(X,\mathbb Z)$$
implies $\rk(\mathcal T_{\Bl_{T'}X'})=21$. On the other hand
$$\mathcal T_{\Bl_{R'}W'}\simeq\mathcal  T_{\Bl_{U} \p^4}\simeq\mathcal  T_U\simeq\mathcal  T_{U''}$$
 yields $$\rk(\Pic(U''))=22-\rk(\mathcal T_{U''})=22-\rk(\mathcal T_{\Bl_{R'}W'})=22-21=1.$$ If $H''\subset U''$,  let $h''=[H'']\in H^2(U'',\mathbb Z)$.
Then $h''^\perp=\mathcal T_{U''}\subset H^2(U'',\mathbb Z)(-1)$  and  the previous isomorphisms defined via  \eqref{blU} induce an isomorphism 
$$H^4(X,\mathbb Z)\supset \langle h^2,[S_{38}]\rangle^\perp \stackrel{\simeq}\longrightarrow  \langle h''\rangle^\perp\subset H^4(U'',\mathbb Z)(-1)$$ respecting  Hodge structures. Therefore the K3 surface $U''$ is associated to $X$ in the sense of Hodge Theory according to Hassett, see \cite[Section 3.2]{Levico}.

The isomorphism between $H^4(\Bl_{T'}X',\mathbb Z)$ and $H^4(\Bl_{R'}W', \mathbb Z)$ sends the classes corresponding to the 10   exceptional lines on $S_{38}$ to the classes corresponding to the ten exceptional lines on $U\subset\p^4$ while the rational normal curve of degree four on $U$ correspond to the class of  $H^2(T',\mathbb Z)(-1)$, $T'\simeq \p(\Oo_{\p^1}(4)\oplus\Oo_{\p^1}(4))\to  \p^1$,  not contacted by $\omega$. This gives an  interpretation of the  fact that $\tilde\phi(T')=C=\tilde\psi(R')$ is a rational normal curve of degree four (or, equivalently, that $T'$ admits a section which is a rational normal curve of degree four) and shows that the exceptional curve of degree four $\tilde E\subset U$ is exactly $\nu(R')\cap U$.

\section{A geometric method for detecting the exceptional $(-1)$--curves on some non minimal K3 surfaces}\label{geomet}

We shall now consider the problem of finding the non minimal $K3$ surface $U\subset W$ in the base locus scheme of
a the birational map $\mu:X\map W$ defined in the previous sections for a general cubic $[X]\in \mathcal C_{d}$ with $d=14, 38$ in order
to develop a method to be applied later in the more difficult case $d=42$.
The detection of the $(-1)$-curves on the (smooth non minimal K3) surface $U$ (or on its linear normalization if $U$ has
nodes) is a  delicate and intriguing problem, which in some cases can lead to  the explicit construction of the general K3 surface of degree $2g-2$ and genus $g$. When this is possible,  one usually gets  a direct proof that the corresponding moduli space of polarized K3 surfaces is unirational. In the previous section to analyse  the case $d=38$, following the traditional approach of \cite{DES}, we used the smoothness of $U\subset\p^4=W$ and appealed to the 
Le Barz Formulas together with Adjunction Theory in order to find the ten exceptional lines on $U$. Then, after the contraction of these ten  $(-1)$-curves,  there appeared a last exceptional curve, which is a quartic rational normal curve already contained in $U$.

It is very difficult to try to adapt these arguments to surfaces $U\subset W$ with $W$ a Fano fourfold lying in spaces of higher dimension. So we elaborated an entirely new geometric method, which works efficiently in the cases $d=14, 26, 38, 42$ to prove that $U$ is the blow-up of  a  K3  surface of degree $d$ and genus $g=\frac{d+2}{2}$ {\it associated} to $X\in\mathcal C_d$. Our analysis is based on the existence of the congruence, of the map $\mu:\p^5\map W$ and on a careful study of the known examples studied by Fano in \cite{Fano} (see also \cite{BRS}).

Recall that by definition of congruence of $(3e-1)$-secant curves of degree $e\geq 1$ to $S\subset\p^5$, we have a diagram~\eqref{diagramma congruenza},
where $\pi:\mathcal D\to \mathcal H$
is the  universal family over the parameter space $\mathcal H$ and where  $p:\mathcal D\to\p^5$ is  the tautological morphism,
which by definition is birational, see Subsection \ref{nec}.
The {\it Fundamental Locus of the congruence} 
$$E=\overline{\{q\in\p^5\:\:\dim(p\inv(q))>0\}}\subset\p^5$$
is the base locus of $p\inv$ by Zariski Main Theorem, which also implies that $E$ is  the image of the ramification locus of $p$. The locus $E$ has codimension at least two. Let $E_1,\ldots, E_r$, $r\geq 1$ be the irreducible components of dimension three of $E$, if any.  Suppose there exists the associated rational map $\mu:\p^5\map W$ defined by the linear system $|H^0(\mathcal I_S^e(3e-1))|$ such that a general fiber of $\mu$ is a curve of the congruence and let $C_i=\mu(E_i)\subset W$, $i=1,\ldots,r$. These are special curves in $W$, defined via  the congruence but without any apparent relation with cubic fourfolds through $S$.  

Let us start by describing the surfaces considered in Example \ref{esempiFano} from this perspective.

\begin{ex}\label{esdP5} Let $S=\Bl_{\{p_1,\ldots, p_4\}}\p^2\subset\p^5 $ be a smooth quintic del Pezzo surface. The embedding is given by the linear system of cubics passing through the four points $p_1,\ldots, p_4$, denoted as usual by $|3L-E_1-E_2-E_3-E_4|=|-K_S|$. The pencil of lines through one of the points $p_i$, $i=1,\ldots, 4$, and the pencil of conics through $p_1,\ldots, p_4$ produce five base point free  pencils of conics on $S$.  Moreover, by considering a Cremona transformation centered at three of the four base points we can always find a representation of $S$ as the blow-up at four points in which any of the five  pencil of conics is represented by the pencil of conics through the four base points. 

Let $C\subset S$ be an irreducible conic and  $\Pi=\langle C\rangle=\p^2$ be its linear span. Since $S$ is defined by quadratic equations $\Pi\cap E_i$ is either empty or consists of a point (otherwise $\Pi\cap S$ would contain $C$ and the line $E_i$). Let $\pi:S\to \p^2$ be the blow-up morphism.
Then $C=dL-\sum_{i=1}^4 a_iE_i$ with $0\leq a_i\leq 1$. From $2=3d-\sum_{i=1}^4a_i$, we get $3d=2+\sum_{i=1}^4a_i\leq 6$, yielding either: $d=1$ and $a_i=1$ for only one index $i$; or $d=2$ and $a_i=1$ for every index $i$. In conclusion $C$ belongs to one of the five pencils described above. 

Let $\mu:\p^5\map\p^4$ be the map associated to the congruence of secant lines to $S$. By Proposition \ref{verme} the closure of the fibers of $\mu$ are either secant lines to $S$ or planes cutting $S$ along a conic. Hence the Fundamental Locus $E$ of the congruence of secant lines to $S$ is the locus of the planes spanned by conics through $S$ (through a point of the plane there passes infinitely many secant lines to $S$ and if through a point of a secant line there passes another secant line, these lines span a plane by the analysis of the fibers of $\mu$). 

Let $\pi_1:S\to \p^1$ be a morphism induced by one of the five pencil of conics and let $\pi:S\to \p^2$ be a blow-up morphism in which the pencil is represented by conics through the four base points. Then $\pi_1^*(\Oo_{\p^1}(1))=\Oo_S(2L-E_1-E_2-E_3-E_4)$  and $\pi^*(\Oo_{\p^2}(1))=\Oo(L)$ so that $\pi_1^*(\Oo_{\p^1}(1))\otimes\pi^*(\Oo_{\p^2}(1))=\Oo_S(3L-E_1-E_2-E_3-E_4)$. By the universal property of the product $\alpha=\pi_1\times\pi:S\to E_1=\p^1\times \p^2\subset\p^5$ is such that $\alpha^*(\Oo_{\p^5}(1) )=\Oo_S(1)$. Hence $\alpha$ is an embedding and $S\subset E_1$ is a divisor of type $(1,2)$. In conclusion  we have five  Segre 3-folds $E_i\simeq\p^1\times\p^2$, $i=1,\ldots, 5$ containing $S$ as a divisor of type $(1,2)$ and $E=E_1\cup\cdots\cup E_5$. Then $C_i=\mu(E_i)\subset\p^4$ are five lines because $\mu$ is defined by the linear system of quadrics through $S$ (remark that by restricting $\mu$ to $E_i$ we have $(2,2)-(1,2)=(1,0)$ on $E_i$). 
\end{ex}

\begin{ex}\label{essc4}
If $S\subset\p^5$ is a general quartic rational normal scroll, the Fundamental Locus $E$ of the congruence of secant lines to $S$ consists of the unique Segre 3-fold $\Sigma\simeq \p^1\times \p^2$ containing $S$ as a divisor of type $(0,2)$. Indeed, $S\simeq \p^1\times \p^1$ embedded in $\p^5$ by $\Oo_{\p^1\times\p^1}(1,2)$ so that, reasoning as above,  $S\simeq \p^1\times C\subset \p^1\times \p^2=\Sigma\subset\p^5$ with $C\subset\p^2$ a conic. The conics on $S$ are only the fibers of the first projection $\p^1\times C\to \p^1$. The fibers of  $\mu:\p^5\map W$ with $W\subset\p^5$ a smooth quadric hypersurface are all linear by Proposition \ref{verme}. Let $E\subset\p^5$ be the Fundamental Locus of the congruence of secant lines to $S$. If $q\in E$, then through $E$ there pass at least two secant lines to $S$. So $\Pi_q=\overline{\mu\inv(\mu(q))}$ is a plane cutting $S$ along a conic, yielding $E\subset \Sigma$ and hence $E=\Sigma$. Then  $C=\mu(E)\subset W$ is a smooth conic because on $E$ we have $(2,2)-(0,2)=(2,0)$.
\end{ex}

In general it is not easy  to determine the $E_i$'s and then calculate  their images (although possible in all the examples treated here). We shall now present a  key remark,  which allows us to determine the $C_i$'s without necessarily computing the $E_i$'s in all the known examples we studied until now.

\begin{meth}\label{geometh}
Suppose there exists a congruence of  $(3e-1)$-secant curves of degree $e\geq 1$ to $S\subset\p^5$ with Fundamental Locus $E=E_1\cup\cdots \cup E_r$ with $E_i\subset E$ the irreducible components of dimension three of $E$, if any.   Suppose there exists the associated rational map $\mu:\p^5\map W$ defined by the linear system $|H^0(\mathcal I_S^e(3e-1))|$ such that a general fiber of $\mu$ is a curve of the congruence and let $C_i=\mu(E_i)\subset W$, $i=1,\ldots,r$. Since through a general point of $E_i$ the fiber of $\mu$ has dimension at least two, we expect that $C_i\subset W$ is a curve, see Examples \ref{esdP5} and \ref{essc4}. 

Let  $X_j\subset\p^5$, $j=1,2$, be a general cubic through $S$ and let
$U_j\subset W$ be the associated surface contained in the base locus of the inverse of the restriction of $\mu$ to $X_j$. Then $C_i=\mu(X_j\cap E_i)\subset U_j$ for every $i=1,\ldots, r$ and for $j=1,2$, so that
$$C_1\cup\cdots\cup C_r\subset U_1\cap U_2.$$
Since $U_1$ and $U_2$ are moving surfaces in the fourfold $W$, one expects that $C_1\cup\cdots\cup C_r$ is exactly the one dimensional component $C$ of $U_1\cap U_2$. 

From  the equations of  $U_1$ and $U_2$ we derive immediately those defining  $C$. In the cases under consideration this allows us to verify the smoothness of $C$, that $C=C_1\cup\cdots\cup C_r$ and that all the disjoint irreducible components $C_i$'s are rational (many components are lines) and that $h^1(N_{C/U})=0$.  The knowledge of the equations of  the $U_j$'s yields $p_g(U_j)=1$, $q(U_j)=0$ (at least in the smooth cases) by direct computation. Collecting all the information  one proves that $U$ is the blow-up of a K3 surface as shown by the next crucial remark.

\begin{lema}\label{lemadjunction} Let $U$ be a smooth projective surface with $p_g(U)=1$ and with $q(U)=0$. Let $C=C_1\cup\cdots \cup C_r$, $r\geq 1$,  be a smooth curve on $U$ with $C_i$ a rational curve for every $i=1,\ldots, r$. Then:
\begin{enumerate}
 \item $C_i^2<0$ for every $i=1,\ldots r$.
 \item  Every $C_i$ is a $(-1)$-curve on $U$ if and only if $h^1(N_{C/U})=0$.
\item If $h^1(N_{C/U})=0$ and if there exists an ample divisor $H$ on $U$ such that $H\cdot K_U=H\cdot C$, then $K_U=C_1+\cdots +C_r$ and $U$ is the blow-up at $r$ distinct points of a K3 surface $U'$. 
\item Under the hypothesis in (3), if $\pi:U\to U'$ is the blow-up,  if  $H'=\pi(H)$,  then $$H=\pi^*(H')-\sum_{i=1}^r(H\cdot C_i)C_i.$$
\end{enumerate}
\end{lema}
\begin{proof} From $q(U)=h^1(\mathcal O_{U})=0$ we deduce $h^0(\Oo_{U}(C_i))=h^0(\Oo_{\p^1}(C_i^2))+1$. If $C_i^2\geq 0$, then the rational curves in $|H^0(\Oo_{U}(C_i))|$ would cover $U$ and $U$ would be uniruled contradicting  $p_g(U)=1$. So $C_i^2<0$ for every $i=1,\ldots, r$.

The $C_i$'s are disjoint smooth rational curves,  $\Oo_C(C)_{|C_i}\simeq \Oo_{\p^1}(C_i^2)$ and  
$$h^1(N_{C/U_j})=\sum_{i=1}^r h^0(\Oo_{\p^1}(-2-C_i^2))$$ is equal to $0$ if and only if $C_i^2=-1$ for each $i=1,\ldots, r$, proving (2). 

Since $p_g(U)=1$, the canonical class is effective. From $K_{U}\cdot C_j=-1$ we deduce $K_U=C+D$ with $D\geq 0$. Then $H\cdot C=H\cdot K_U= H\cdot C+H\cdot D$ yields $H\cdot D=0$ and hence $D=0$.  Since the $C_i$'s are disjoint $(-1)$-curves, the contraction of the curves in $C$ produces a smooth surface $U'$ and a morphism $\pi:U\to  U'$ which is the blow-up of $r\geq 1$ distinct points on $U$. Then $K_{U'}=0$ and $q(U')=q(U)=0$ imply that $U'$ is a K3 surface. The last claim about the expression of $H$ via pull-back of $H'$ is now obvious.
\end{proof}
\end{meth}

\begin{ex} ({\bf Application to smooth quintic del Pezzo surfaces in $\p^5$})
Suppose $S\subset\p^5$ is a smooth del Pezzo surface of degree 5. Then $U_j\subset\p^4$, $j=1,2$, are smooth surfaces of degree 9 and sectional genus 8 with  $p_g(U_j)=1$ and $q(U_j)=0$ (the invariants are determined via the explicit equations obtained via the restriction of $\mu:\p^5\map \p^4$). The one dimensional component $C$ of $U_1\cap U_2$ is a smooth curve of degree 5 and arithmetic genus $-4$ from which it immediately follows that $C$ has five irreducible components and hence  is the union of five distinct lines $C_1,\ldots, C_5$, as we already  know. 
   
Let $U=U_1$. After verifying  that $h^1(N_{C/U})=0$, we deduce that every $C_i$ is a $(-1)$-curve on $U$.   Since   $H\cdot K_U=-H^2+2g(H)-2=-9+16-2=5=H\cdot C$ we conclude by Lemma \ref{lemadjunction} that $\pi:U\to U'$ is the blow-up at five distinct points $p_1, \ldots, p_5$ of the K3 surface $U'$ and that  $H'=\pi_*(H)\subset U'$  is a very ample divisor on $U'$ such that
$(H')^2=9+5=14$, $g(H')=8$. Indeed, $H'$ is ample and the generality of $X_1$ and of $U=U_1$ imply that $\rk(\Pic(U'))=1$ (see the argument used  at the end of Subsection \ref{s38} or compute the moduli of the K3's). Hence 
$U'\subset\p^8$ is a $K3$ surface of degree 14 and genus 8 and the linear system $|H|$ on $U$  corresponds to the  hyperplane sections of $U'$ passing through $p_1,\ldots, p_5$. In particular, we see that these points  impose only four independent conditions to hyperplane sections.
\end{ex}

\begin{ex} ({\bf Application to general quartic rational normal scrolls in $\p^5$}) For $S\subset\p^5$ a general rational normal scroll, the corresponding surfaces $U_j\subset W$ are smooth, have degree 10, sectional genus 7, $p_g(U_j)=1$ and $q(U_j)=0$. The one dimensional component $C$ of $U_1\cap U_2$ is a smooth conic, which is a $(-1)$-curve on $U_j$ because $h^1(N_{C/U_j})=0$. Since $H\cdot K_{U_j}=2=H\cdot C$, we  deduce $K_{U_j}=C$ and we can apply Lemma \ref{lemadjunction}. We conclude that $U_j\subset\p^5$ is  
obtained by blowing-up a K3 surface $U'\subset\p^8$ of degree 14 and genus 8 at one point $p\in U'$ and  that $|H|$ corresponds to the linear system of hyperplanes sections having a point of multiplicity at least two  at $p$.
\end{ex}

\begin{ex} ({\bf Application to degree 10 surfaces  $S_{38}\subset \p^5$})
The surfaces $U_j\subset\p^4$, $j=1,2$, are smooth, have degree 12, sectional genus 14, $p_g(U_j)=1$ and $q(U_j)=0$, see Subsection \ref{s38}.
The one dimensional component $C$ of $U_1\cap U_2$ is a smooth curve of degree 14, of arithmetic genus $-10$ with eleven irreducible components.
By projecting $C$ generically to $\p^3$ (to speed computations) and to $\p^2$ (to read more efficiently the decomposition), we get a smooth curve of degree 14, respectively a plane curve $\tilde C$ of degree 14. The lines in $\tilde C$ disappear by taking duality, that is the image of $\tilde C$ via the Gauss map of $\tilde C$. Then, by reflexivity, the bidual curves of $\tilde C$ consists of the irreducible curves in  $\tilde C$ different from lines.
It turns out that  the bidual curve of $\tilde C$ is a quartic curve with three nodes from which it  follows that $C$ is the disjoint union of 10 lines and of a quartic rational normal curve, see also the ancillary file 
\verb!code_section_6.m2!. In conclusion,  $C$ is the union of eleven smooth rational curves. 

The eleven curves $C_i$'s are $(-1)$-curves on $U_j$ because $h^1(N_{C/U_j})=0$, see \verb!code_section_6.m2!. for the computation and then apply Lemma \ref{lemadjunction}. Since $H\cdot K_U=-H^2+2g(H)-2=-12+28-2=14=H\cdot C$, we deduce from Lemma \ref{lemadjunction} that there exists a  contraction  $\pi:U\to U'$  of the eleven $(-1)$-curves $C_i$'s  to the eleven distinct  points $p_1,\ldots, p_{11}$ on the smooth K3 surface $U'$. Then  $H'=\pi_*(H)\subset U'$ is an ample divisor on $U'$ such that
$(H')^2=12+16+10=38$ and  $g(H')=14+6=20$.  For a general $X$ through $S_{38}$ we have that 
$U'$ is a  K3 surface of degree 38 and genus 20 with $\rk(\Pic(U'))=1$ (see the argument used  at the end of Subsection \ref{s38} or compute the moduli of the K3's) so that $|H'|$ gives an embedding $U'\subset\p^{20}$. The linear system $|H|$ on $U$  corresponds to the  hyperplane section of $U'$ passing through $p_1,\ldots, p_{10}$ and having a point of multiplicity at least four at $p_{11}$. In particular, we see that these points do not impose independent conditions to hyperplane sections of $U'$. 
\end{ex}

\section{Rationality of cubics in \texorpdfstring{$\mathcal C_{42}$}{C42} via congruences of 8-secant twisted cubics}\label{C42}  

In this section,
we first construct a family of  surfaces in $\p^5$  to describe  the divisor $\mathcal C_{42}$ as the locus of cubic fourfolds containing these surfaces. This family  is related to 
an example of a surface of degree 9 and sectional genus 2 contained in a del Pezzo fivefold, which has been discovered (computationally) 
for the first time in \cite{HS}.
 Here we give an explicit and geometric description of the complete family of these surfaces 
 inside a del Pezzo fivefold and then use them to construct surfaces in $\PP^5$ of degree 9 and sectional genus 2 with five  nodes.
Finally we use this new description of $\mathcal C_{42}$ to show that 
$\mathcal C_{42}$ is unirational (see Corollary~\ref{uniraz C42}) 
and that every cubic fourfold in $\mathcal C_{42}$  is rational (see Theorem~\ref{C42_Flop}).

\subsection{Birational representations of  del Pezzo fivefolds}\label{birdP5}

A del Pezzo fivefold $V\subset\p^8$ is  a smooth hyperplane section of $\mathbb G(1,4)\subset\p^9$. The following result is well known in classical Algebraic Geometry (see \emph{e.g.} \cite{Todd} and \cite[Section~10]{Semple}).

\begin{Proposition}\label{birVS12} Let $V=\mathbb{G}(1,4)\cap\p^8\subset\p^8$ be a del Pezzo fivefold and let $\Pi\subset V$ be
a plane with class $\sigma_{2,2}$ in $\mathbb{G}(1,4)$. Then:
\begin{enumerate}[(i)]
\item  the projection from $\Pi$ restricted to $V$ induces a birational map $V\map\p^5$, whose base locus scheme is $\Pi;$
\item the inverse map $\p^5\map V\subset\p^8$ is given by the linear system 
of quadric hypersurfaces through a rational normal cubic scroll contained in a hyperplane in
$\p^5$.
\end{enumerate}
\end{Proposition}

\begin{Remark}\label{F1F2}  A del Pezzo fivefold $V=\mathbb{G}(1,4)\cap\mathbb{P}^8\subset\p^8$ contains two distinct families of planes, $\mathcal F_1$ and $\mathcal F_2$,  which  in the Chow ring of $\mathbb{G}(1,4)$ have Schubert
 classes given by $\sigma_{2,2}$ and $\sigma_{3,1}$.
The family $\mathcal F_1$  has dimension 3 and through a point of $V$ there passes a unique plane of the family.   The family $\mathcal F_2$ has dimension 4 and through a general point of $V$ there passes a one dimensional family of these planes.

If  $\alpha:\p^5\map V\subset\p^8$ denotes the birational map 
 defined by the quadrics through a rational normal cubic scroll $\Delta\subset H\simeq\PP^4\subset\PP^5$ 
 as in Proposition~\ref{birVS12}, then
  the planes through a point $q=\alpha(p)$, $p\in\p^5\setminus H$, correspond, respectively, to the unique plane
 generated by $p$ and the directrix line of $\Delta$ and to the planes generated by $p$ and by a line of the ruling of $\Delta$.
  
Since $N_{C/V}\simeq \Oo_{\p^1}(1)^{\oplus 2}\oplus\Oo_{\p^1}(2)^{\oplus 2}$ (see for example \cite[p. 403]{IL}) for an irreducible  conic $C\subset V\subset\p^8$, the Hilbert scheme $\mathcal Con$ of conics contained in $V$ has dimension $10=h^0(N_{C/V})$,
 a fact that can be also deduced from a simple parameter count using the previous birational representation of $V$.
 \end{Remark}

The next result is also classical and well-known. It can also be easily verified via an explicit computation.

\begin{Proposition}\label{birV} Let $V=\mathbb{G}(1,4)\cap\mathbb{P}^8\subset\p^8$ be a del Pezzo fivefold 
and let $C\subset V$ be an irreducible conic such that its linear span $\Pi=\langle C\rangle$ is not contained in $V.$ Then:
\begin{enumerate}[(i)]
\item  the projection from $\Pi$ restricted to $V$ induces a birational map $V\map\p^5$, whose base locus scheme is $C$;
\item the inverse map $\p^5\map V\subset\p^8$ is given by the linear system of cubic hypersurfaces vanishing on a rational scroll in
$\p^5$ of dimension three and degree four, which is a  projection of a smooth quartic rational normal
scroll in $\p^6$.
\end{enumerate}
\end{Proposition}

\subsection{Curves of degree 8 in $\p^5$ with a node and with geometric genus  2  contained in a projected rational scroll of degree four}\label{curves82}
Now we prove some geometrical properties about irreducible curves of degree $8$ and arithmetic genus $3$ in $\p^5$. {We shall restrict 
ourselves to the case of curves with a node, used in the sequel, although the same proof works also for smooth curves (this case has been  considered in  \cite{ST}).}
\begin{Proposition}\label{cidp5}  Let $C\subset\p^5$ be a non degenerate curve of degree 8 and arithmetic  genus $3$ with a node. Then:
\begin{enumerate}[(i)]
\item   the curve $C\subset\p^5$ is the complete intersection of a pencil of quintic del Pezzo surfaces on a Segre 3-fold $\Sigma=\p^1\times\p^2\subset\p^5$ such that the general element of the pencil is smooth.

 \item The curve  $C\subset \p^5$   has ideal generated by seven quadratic forms, defining a birational map $\psi:\p^5\map W\subset\p^6$ onto a quartic hypersurface and such that $\psi(\Sigma)=Q\subset W\subset\p^6$ is a smooth quadric surface contained in the base locus scheme of $\psi^{-1}$. 
 
 \item The preimages on $\Sigma$ of the lines of the two rulings of $Q$ are, respectively, the planes of the ruling of $\Sigma$ and the pencil of del Pezzo surfaces through $C$. 
 
 \item The map $\psi$ is an isomorphism outside $\Sigma\cup\Sec(C)$ and $\Sec(C)$ is mapped onto a degree 24 surface $T$. The quartic hypersurface $W$ has double points along $Q$ and $T$ and $\psi^{-1}:W\map\p^5$ is given by the restriction of a linear system of quartic hypersurfaces in $\p^6$ passing simply through $T$ and having double points along $Q$.
\end{enumerate}
\end{Proposition} 
\begin{proof} Let  $C\subset \p^5$ be an irreducible curve of degree 8  and geometric genus 2 with a node.
Let  $\nu:C'\to C$ be its normalization and let $\mathcal L=\nu^*(\Oo_C(1))$. The linear system $|H^0(\mathcal L)|$ has dimension six by Riemann-Roch,  it is very ample and it embeds $C'$ in $\p^6$ as a smooth curve of degree 8 and genus 2. The curve $C\subset \p^5$ is the projection of $C'$ from a  point $q$ on the secant variety $\Sec(C')\subset\p^6$ of $C'$ but not belonging to the tangential surface $\Tan(C')$.

Let $P'=\langle p'_1,\ldots, p'_4\rangle$, $p'_i\in C'$, be a four secant $\p^3$ passing through $q$ and let $P=\langle p_1,\ldots, p_4\rangle$, with $p_i\in C$ the projection from $q$ of $p'_i$. Then $P$ is a four secant plane to $C$ and,  letting $D=p_1+\cdots+p_4$ be the corresponding Cartier divisor on $C$, a general hyperplane through $P$ cuts $C$ in $D$ and in other 4 points $p_5,\ldots, p_8$ such that $\Pi=\langle p_5,\ldots, p_8\rangle=\p^3$. Letting $E=p_5+\cdots +p_8$, from the previous definitions and from Riemann-Roch we deduce  $h^0(\O_C(D))=2$ and $h^0(\O_C(E))=3$. The linear system $|H^0(\O_C(D))|$ defines a morphism $\xi:C\to \p^1$ of degree 4 while $|H^0(\O_C(E))|$ defines a  morphism $\eta:C\to  \p^2$ birational onto its image. By the universal property of the product we have a morphism $\xi\times\eta:C\to \p^1\times\p^2$ which composed with the Segre embedding $\p^1\times\p^2\subset\p^5$ gives 
a morphism $\epsilon:C\to\p^5$ such that $\epsilon^*(\O_{\p^5}(1))=\O_C(D)\otimes\O_C(E)=\O_C(1)$. This means that the embedding $C\subset\p^5$ factors through $\Sigma\simeq\p^1\times\p^2\subset\p^5$ in such a way that $\xi$ is the composition of $\epsilon$ with the projection onto the first factor and that  
$\eta$ is the composition of $\epsilon$ with  projection onto the second factor. In particular $P=p\times \p^2\subset\Sigma$ for some $p\in\p^1.$ By Riemann-Roch we deduce $h^0(\mathcal I_C(2))\geq 7$ so that $h^0(\mathcal I_{C\cup P}(2))\geq 5$ (from  $P\cap C=\{p_1,\ldots, p_4\}$ we deduce that  to contain $P$  imposes only two conditions to quadrics vanishing on $C$). Since $h^0(\mathcal I_\Sigma(2))=3$, there exists at least a pencil of quadrics $\{Q_{\lambda}\}_{\lambda\in\p^1}$ vanishing  on $C\cup P$ but not on $\Sigma$. This pencil of quadrics cuts $\Sigma$ along $P$ and along a residual pencil of divisor $\{S_{\lambda}\}_{\lambda\in\p^1}$
of type $(1,2)$ containing $C$. The projection from $P$ maps $\Sigma$ onto a plane $\Pi$ and $C$ onto a plane quartic curve.
Hence $C$ is not contained in any divisor of type $(0,2)$ on $\Sigma$ (they project from $P$ onto a conic) and being non degenerate in $\p^5$ is not contained in any divisor of type $(1,1)$ or $(1,0)$ or $(0,1)$, proving that  every  $S_{\lambda}$ is irreducible. Since the complete intersection
of two distinct  divisors of type $(1,2)$ on $\Sigma$ without common irreducible components is a curve of degree 8 and arithmetic genus 3, we conclude that $C$ is the  complete intersection of two  surfaces in the pencil $\{S_{\lambda}\}_{\lambda\in\p^1}.$ { The projection from $P$ restricted to $S_{\lambda}$ resolves to a birational morphism $\sigma_{\lambda}:S_{\lambda}\to\Pi$. The planes in  $\Sigma$ cuts on  $S_{\lambda}$ a pencil of conics, whose image by $\sigma_{\lambda}$ is a pencil of conics having a base locus scheme $Z_{\lambda}\subset\Pi$ of length 4.}
{The inverse map $\sigma_{\lambda}\inv:\Pi\map S_{\lambda}$ is given by the linear system of cubics vanishing at $Z_{\lambda}$.  Let $\tilde C\subset\Pi$ denote the projection of $C$ from $P$. Clearly $Z_{\lambda}\subset \tilde C$ and the divisors $\{Z_{\lambda}\}_{\lambda\in\p^1}$ vary in a $g_4^1$ on $\tilde C$. Moreover, the pencil of conics through a fixed  $Z_{\lambda}$ cuts on $\tilde C$ the $g_4^1=|H^0(\O_{\tilde C}((\sigma_{\lambda})_*(D)))|$. On the contrary, fixed a general divisor $\tilde D_\mu$ in the last $g_4^1$, the pencil of conics through $\tilde D_\mu$ cuts on $\tilde C$ the  $g_4^1=\{Z_{\lambda}\}_{\lambda\in\p^1}$. Hence for general $\lambda\in\p^1$ the scheme $Z_{\lambda}$ is smooth by Bertini Theorem and the corresponding $S_{\lambda}\subset\p^5$ is a smooth quintic del Pezzo surface.} Thus,  part (i) is proved.\\
Let $C=S\cap S'\subset \Sigma$ be a curve of degree 8 and arithmetic genus 3, complete intersection of a pencil $\{S_{\lambda}\}_{\lambda\in\p^1}$ of divisors of type $(1,2)$ on $\Sigma$, whose general member is smooth.  Arguing as in \cite[pp. 435--436]{HKS}, an iterated use of the Mapping Cone and the knowledge of the resolution of a $S_{\lambda}\subset\p^5$ yields that $C\subset\p^5$ has ideal generated by seven quadratic equations (more precisely we can obtain the complete free resolution of the ideal of $C\subset\p^5$, see {\it loc. cit.}). 
Let $\psi:\p^5\map\p^6$ be the map defined by $|H^0(\mathcal I_C(2))|.$  The image $W=\overline{\psi(\p^5)}\subset\p^6$ is a quartic hypersurface, see \cite{ST} and \cite{HKS}. The restriction of $\psi$ to $\Sigma$ is a linear system of dimension three because $\Sigma$ is defined by three quadratic equations vanishing on $C$. Let $S\subset \Sigma$ be a quintic del Pezzo surface through $C$ and let $\pi:S\simeq\Bl_{\{q_1,\ldots,q_4\}}\p^2\to \p^2$ be a birational representation such that the pencil of  conics  on $S$ given by the restriction of the projection onto the first factor of  $\Sigma$ is mapped on $\p^2$ to the pencil  of conics through $q_1,\ldots, q_4$. Then $\pi(C)$ is a quartic curve passing through $q_1,\ldots, q_4$. Hence  the free part of the restriction of $|H^0(\mathcal I_C(2))|$ to $S$ is given by the pencil of strict transforms of conics passing through $q_1,\ldots,q_4$ and $\psi(S)$ is a line contained in $Q=\overline{\psi(\Sigma)}\subset W$. In particular, $Q\subset \p^3\subset\p^6$ is a surface.  Since each plane $\Pi\subset\Sigma$ cuts $C$ in four points, $\psi$ maps $\Pi$ onto a line in $Q$ cutting $\psi(S)$ in the point corresponding to $\psi(\Pi\cap S)$. In conclusion $Q\subset \p^3$ is a smooth quadric surface and all the claims in  (ii) and  in  (iii) are proved.  A finer  analysis of the map $\psi$, as in  \cite[Section 4]{ST} or in \cite[Section 3, Lemma 3.3, Lemma 3.4]{HKS}, leads to the description of all the positive dimensional fibers of $\psi$ and of all the properties listed in part (iv). We refer to \cite[p. 208]{ST} for the details on these computations.
\end{proof}

The following result will play a crucial role in our geometric constructions together with the description of the map $\psi$ defined  above.
\begin{Proposition}\label{CB5} Let $B\subset\p^5$ be a rational scroll of dimension 3 and degree 4 which is a  general projection of a smooth quartic rational normal scroll $B'\subset\p^6$. Then:
\begin{enumerate}[(i)]
\item there exists an irreducible   family $\mathcal F$ of dimension 15 of curves of degree 8 and geometric genus 2  on $B$, whose general member is nodal.
\item There exists an irreducible  family $\mathcal D$ of dimension 16 of quintic del Pezzo surfaces in $\p^5$, whose general member is smooth and cuts $B$ along a general curve of the family $\mathcal F$. 
\end{enumerate}
\end{Proposition} 
\begin{proof} Let  $C'\subset\p^6$ be a smooth curve of degree 8 and genus 2. The projection of  $C'\subset\p^6$ from four general points $p'_1,\ldots, p'_4$ on it
 is a quartic plane curve $\tilde C\subset\p^2$ with a node $\tilde p_5$. Let $\tilde p_1,\ldots, \tilde p_4\in \tilde C$ be the images of $p'_1,\ldots, p'_4$. The linear system of quartics curves through $\tilde p_1,\ldots, \tilde p_4$ and having double points at $\tilde p_5$  embeds the blow-up of $\p^2$  at $\tilde p_1,\ldots,\tilde p_5$ as a smooth surface $S\subset\p^7$ of degree  8 and genus 2 having  $C'\subset\p^6$ as a hyperplane section. The linear system $|H^0(\mathcal I_{\{\tilde p_1,\tilde p_2, \tilde p_3,\tilde p_5\}}(2))|$ yields a morphism $\phi_1:S\to \p^1$ while $ |H^0(\mathcal I_{\{\tilde p_4,\tilde p_5\}}(2))|$ yields a morphism $\phi_2:S\to \p^3$. The composition of the morphism $\phi_1\times\phi_2:S\to \p^1\times\p^3$ with the Segre embedding is given by the linear system embedding $S$ in $\p^7$. Hence $S\subset \p^1\times \p^3\subset\p^7$ and $C'\subset\p^6$ is contained in a smooth quartic rational normal scroll $B'\subset\p^6$, which is a hyperplane section of $\p^1\times \p^3$.

Let  
$C'\subset B'$ be a curve  of degree 8 and genus 2.  Then $-K_{B'}\cdot C'=18$,  $\deg(N_{C'/B'})=20$
by Adjunction Formula and $\chi(N_{C'/B'})=20+2(1-g(C'))=18$. Since in an explicit example we verified that $h^0(N_{C'/B'})=18$ and that $h^1(N_{C'/B'})=0$,  the general curve of degree 8 and genus 2 contained in $B'$ belongs to a unique irreducible and generically smooth component of dimension 18 of the corresponding Hilbert scheme.  The family of the secant varieties $\Sec(C')$ with $C'\subset B'$ has dimension $21=18+3$, so through a general point  $q\in \p^6$ there passes a family of dimension 15  of such secant varieties. 
The projection from  $q$ of the corresponding curves  produces a  15-dimensional family of nodal curves of degree 8 and geometric genus 2  on the scroll $B$, projection of $B'$ from $q$. Since one verifies that $h^0(N_{C/B})=15$ (a fact which can be easily computed with \emph{Macaulay2}),   the family $\mathcal F$ of such curves is irreducible, generically smooth and of  dimension 15.

By Proposition \ref{cidp5}, a nodal curve  $C\subset B\subset\p^5$ of degree 8 and geometric genus 2 with a node  is the complete intersection of a pencil of del Pezzo surfaces on a Segre threefold $\Sigma\simeq\p^1\times \p^2\subset\p^5$. Let $C\subset S=S_{\lambda}\subset\Sigma$ with $\lambda\in \p^1$ general and let $\pi:S=\Bl_{\{q_1,\ldots,q_4\}}\p^2\to\p^2$ be
the blow-up morphism. Then we can assume that $\pi(C)=\tilde C\subset\p^2$ is a plane quartic curve with a node passing simply through $q_1,\ldots, q_4$, see the proof of Proposition \ref{cidp5}.
The irreducible threefold $B$ has ideal generated by one quadratic form  and by three cubic forms so that it is scheme theoretically defined by nine cubic forms. To study the scheme theoretic intersection $B\cap S$ we restrict to $S$ the linear system $|H^0(\mathcal I_B(3))|$. Each cubic in this linear system cuts $S$ along a curve $D$ containing $C$ and such that  $\pi(D)=\tilde D$ is a curve of degree nine  having triple points at  $q_1,\ldots, q_4$. Hence $\tilde D=\tilde C+\tilde A$ with $\tilde A$ a quintic curve with double points at $q_1,\ldots, q_4$. The linear system $|\tilde A|$  has dimension eight, it is base point free and very ample, proving that $B\cap S=C$ as schemes for $[C]\in \mathcal F$ general. The linear system  $|\tilde A|$ is equivalent via a quadratic standard Cremona transformation centred at $q_2,q_3, q_4$ to the linear system of quartic curves having a double point at $q_1$ and  simple base points at $q_2,q_3,q_4$. 

 By varying $C$ and recalling that any $C$ is a complete intersection of a pencil of quintic del Pezzo surfaces inside $\Sigma$,   we get a family $\mathcal D$ of dimension 16 of quintic del Pezzo surfaces cutting $B$ scheme theoretically along a nodal curve of degree 8 and genus 2 as above. 
\end{proof}

\subsection{A rational surface of degree 9 and sectional genus 2 with 5 nodes contained in a general cubic fourfold of \texorpdfstring{$\mathcal C_{42}$}{C42}}\label{sez C42 nuova}

We can now construct a 25-dimensional family of smooth surfaces of degree 9 and sectional genus 2 on a del Pezzo fivefold $V\subset\p^8$. This was also achieved in \cite[Lemma~3.1]{HS} via a different construction 
of an explicit example that corresponds to a smooth point in $\mathrm{Hilb}_{V}$ and
which is related to  a K3 surface of genus $11$ in $\PP^{11}$.

\begin{Proposition}\label{S42V} Let $V\subset\p^8$ be a  del Pezzo fivefold. 
\begin{enumerate}[(i)]
\item There exists an irreducible and generically smooth 
family $\mathcal S\subset\mathrm{Hilb}_{V}$ of  dimension 25 of surfaces $S\subset V$ of degree 9 and sectional genus 2, whose general member is  smooth.
\item There exists an irreducible  and generically smooth  family $\mathcal S_{42}\subset\mathrm{Hilb}_{\p^5}$ of dimension 48 of surfaces  in $\p^5$ of degree 9 and sectional genus 2 with 5 nodes, whose general member is  
obtained as the projection from a general plane of the family $\mathcal F_1\subset\mathrm{Hilb}_{V}$
of a general surface of the family $\mathcal S\subset\mathrm{Hilb}_{V}$.
\end{enumerate}
\end{Proposition}
\begin{proof} Let $C\subset V$ be an irreducible conic whose linear span $\Pi=\langle C\rangle$ is not contained in the del Pezzo fivefold $V$, and let 
\begin{equation}\label{mappaAlpha}
\alpha:\p^5\map V\subset\p^8
\end{equation}
be the inverse of the projection from $\Pi$ (see Proposition \ref{birV}).
  Let $\mathcal D$ be the family of quintic del Pezzo surfaces described in Proposition~\ref{CB5}, 
  and let $D\in \mathcal D$ be a general member. The smooth surfaces $S=\alpha(D)\subset V\subset\p^8$ have degree 9 and sectional genus 2 and are obtained from $\p^2$ via the linear system of quintics having four double points (or equivalently via the linear system of quartics with a double point and three simple base points) as shown in the proof of Proposition \ref{CB5}. Since  the Hilbert scheme $\mathcal Con$ of conics in $V$ has dimension 10 (see Remark \ref{F1F2}) and since the conics on  such a surface $S\subset\p^8$ belong to the unique  pencil of conics on $S$ (represented on the plane by the conics through the four base points of the linear system), we deduce that the above surfaces describe a family $\mathcal S\subset\mathrm{Hilb}_V$ of dimension at least
$\dim(\mathcal Con)-1+\dim(\mathcal D)=10-1+16=25$. 
Since we have verified in a specific example of surface $[S]\in\mathcal S$ that $\h^0(N_{S/V})=25$,  we deduce that the family $\mathcal S$ is generically smooth of dimension 25.

 Let $\beta:V\map \p^5$ be the birational map induced by the projection from a plane $P\subset V$ of the family $\mathcal F_1$, see Proposition \ref{birV}.
The image via $\beta$ of a general $S\in \mathcal{S}$ is a surface of degree 9, sectional genus 2, cut out by 9 cubics and  having 5 nodes produced by the intersection
 of $P$ with the secant variety to $S$, a fact which can be  verified by a direct computation via \emph{Macaulay2} by following the first steps of the algorithm  described in Remark \ref{algorithm S42} below  (see also  Subsection \ref{contid42}). 
 
 A surface in $\p^5$ of degree 9, sectional genus 2 and with five nodes of the type constructed above will be denoted by $S_{42}$. Since the cubic rational normal scrolls in $\p^4$ depend on 18 parameters, we deduce that the cubic rational normal scrolls in $\p^5$ depend on 23 parameters.
Each rational normal scroll in $\p^5$ determines a map $\beta$ and hence a plane $P\subset V\subset\p^8$. The projection from $P$ of the 25 dimensional family of surfaces of degree 9 and sectional genus 2 contained in $V$ gives a 25 dimensional family of $S_{42}$. By varying the scroll we deduce that
the surfaces $S_{42}\subset\p^5$ describe a family of dimension at least 48=23+25. 
Since we have verified in a specific example of surface $S_{42}\subset\p^5$ that $\h^0(N_{S_{42}/\p^5})=48$,
 we can deduce that  $\h^0(N_{S_{42}/\p^5})=48$ for a general $S_{42}\subset\p^5$ as above. Since we previously proved that these surfaces depend on at least 48 parameters, we conclude that there exists a unique irreducible component $\mathcal S_{42}$ of the corresponding Hilbert scheme containing the surfaces $S_{42}\subset\p^5$, which is generically smooth of dimension  $\dim(\mathcal S_{42})=48$ and such  that the general element of $\mathcal S_{42}$ is of the kind described above.
\end{proof}

\begin{Remark}\label{cs42}  
Let $V=\mathbb{G}(1,4)\cap\p^8\subset\PP^8$ be a  del Pezzo fivefold, 
let $\mathcal F_1\subset \mathrm{Hilb}_{V}$ be the $3$-dimensional family of planes in $V$ with class $\sigma_{2,2}$,
and let 
$\mathcal S\subset\mathrm{Hilb}_{V}$
and 
$\mathcal S_{42}\subset\mathrm{Hilb}_{\p^5}$ 
be, respectively,  the $25$-dimensional family of smooth surfaces in $V$ 
 of degree $9$ and sectional genus $2$, 
and 
the $48$-dimensional family of $5$-nodal surfaces in $\PP^5$
of degree $9$ and sectional genus $2$ constructed in Proposition~\ref{S42V}.
To a general pair $([S],[P])\in \mathcal S\times \mathcal F_1$, we associate a general $[S_{42}]\in\mathcal S_{42}$
defined by $S_{42} = \overline{\beta_P(S)}$, where $\beta_P:V\dashrightarrow\PP^5$
denotes the projection from $P$.
The inverse map $\alpha_{\Delta}=\beta_P^{-1}:\PP^5\dashrightarrow V$
is defined by the quadrics through a rational normal cubic scroll 
$\Delta\subset\PP^4\subset\PP^5$,
which intersects $S_{42}\subset\PP^5$ in a curve $C\subset\PP^5$ of degree $9$ and arithmetic genus $7$.
We now illustrate how one can determine this scroll $\Delta$ from 
the surface $S_{42}$,
and thus determine the pair $(S,P)$ via $S=\overline{\alpha_{\Delta}(S_{42})}$ and $P = \mathrm{Bs}(\alpha_{\Delta}^{-1})$.
 Let $\varphi:\p^5\map Z\subset\p^{8}$ be the rational map defined 
by the linear system $|H^0(\mathcal I_{S_{42}}(3))|$ 
of cubic hypersurfaces through the surface $S_{42}\subset\PP^5$.
Then  $\varphi$ 
is a birational map onto its image $Z\subset\p^{8}$ and 
the base locus of the inverse map $\varphi^{-1}:Z\dashrightarrow\PP^8$ 
is an irreducible surface $T$ with an immersed point $q\in T$.
Then $\overline{\varphi^{-1}(T)}$ is a threefold ruled by trisecant lines to $S_{42}$
while $\overline{\varphi^{-1}(q)}$ coincides with the scroll $\Delta$ intersecting $S_{42}$ along a curve of degree 9 and arithmetic genus $7$.\footnote{The idea of gluing a cubic scroll $\Delta\subset\PP^5$ with another surface $B\subset\PP^5$
along some curve $\Delta\cap B$ and then of  considering the image $\overline{\alpha_{\Delta}(B)}\subset V$  has been systematically applied in \cite{StaDP5} to construct other types of surfaces in a del Pezzo fivefold.} The previous description of $\phi$ assures that a general $S_{42}\subset\p^5$ satisfies Assumption~\ref{assumption1} in Subsection \ref{assdef} and that it has the expected trisecant behaviour.
\end{Remark}

\begin{Remark}\label{algorithm S42} 
 Our construction of the surface $S_{42}$ 
 can be easily implemented and executed in a computer program such as \emph{Macaulay2.}    For the convenience of the reader, we now summarize the algorithm for the construction
  of the general surface in the family $\mathcal S_{42}$ (see also Subsection~\ref{contid42}).
\begin{itemize}
\item  
Let  $p_1,\ldots, p_5 \in \p^2$ be general points. The image 
 of the rational map defined 
 by the linear system $|H^0(\mathcal I_{\{p_1^2,p_2,p_3,p_4,p_5\}}(4))|$
 of plane quartic curves having a double point at $p_1$ and simple base points at $p_2,p_3,p_4,p_5$ 
 is a smooth surface $T\subset\PP^7$ of degree $8$ and sectional genus 2.
 This surface $T$ can be embedded into 
 $\PP^1\times\PP^3\subset\PP^7$ 
  via the product of linear systems
 $|H^0(\mathcal I_{\{p_1,p_3,p_4,p_5\}}(2))|\times |H^0(\mathcal I_{\{p_1,p_2\}}(2))|$. 
 \item 
  Let $L$ be a general secant line to $T$, $H\supset L$ a general hyperplane through $L$,
  and $p\in L$ a general point. Then the projection of $C'=H\cap T\subset H=\PP^6$ to $\PP^5$ from $p$ 
   yields a one-nodal curve $C\subset\PP^5$ of degree 
  $8$ and arithmetic genus $3$ contained in a singular quartic scroll threefold $B\subset\PP^5$, projection
  from $p$ of $B'=\p^1\times \p^3\cap H$.

  \item 
  The quadrics through $C$ define a birational map $\psi$ from $\PP^5$ into a 
  quartic hypersurface  $W\subset\PP^6$. 
  The exceptional locus of $\psi$ contains  
  a Segre threefold $\Sigma\simeq\PP^1\times\PP^2\subset\PP^5$, which is sent into a smooth quadric surface $Q\subset\PP^6$. The quadric $Q\subset W$ can be detected as the unique irreducible component of the base locus scheme of $\psi\inv$ along which the base locus scheme is not generically reduced.

  \item 
  The inverse images via $\psi$ of the two  lines of $Q$ passing 
  through a general point of $Q$ are: a plane of the ruling of $Z$
  and a smooth quintic del Pezzo surface $D\subset\PP^5$, which intersects $B$ along $C$. (Note that $\Sigma$ and the quadric $Q=\overline{\psi(\Sigma)}$ are rational over their field of definition. Indeed,  
  the syzygy matrix of the $3$ quadrics defining a Segre threefold $\Sigma\subset\PP^5_{K}$
  is a $2\times3$ matrix of linearly independent linear forms on $\PP^5_{K}$. 
  These $6$ linear forms can be used to define an automorphism of $\PP^5_{K}$ 
  that sends $\Sigma$ into the Segre embedding of $\PP^1_{K}\times\PP^2_{K}$ in $\PP^5_{K}$.)
  
  \item 
   Finally, the 
   map $\alpha:\PP^5\dashrightarrow V$ defined in \eqref{mappaAlpha} induces 
   an isomorphism 
  between $D$ and 
   a smooth surface $S$ of degree $9$ and 
   sectional 
   genus $2$.  The map  $\beta:V\dashrightarrow\PP^5$ defined in Proposition \ref{birVS12} induces a birational morphism
   from $S$
   to our surface $S_{42}$.
   \end{itemize}

\end{Remark}

We are now ready to prove that a general cubic in $\mathcal C_{42}$ contains a general surface $S_{42}\subset\p^5$ of degree 9 and genus 2 with five nodes in the irreducible family $\mathcal S_{42}$ described above.

\begin{thm}\label{C42S42}
 The irreducible divisor $\mathcal{C}_{42}$ parameterizing cubic fourfolds of discriminant $42$ 
 coincides with the closure of the locus of  cubic fourfolds containing a rational surface 
 $S_{42}$ of degree $9$ and sectional genus $2$ with $5$ nodes of the irreducible family $\mathcal S_{42}$ constructed above. 
 \end{thm}
\begin{proof} Let $h$ be the class of a hyperplane section of $X$, let $h^2$ be the class of  2-cycles $h\cdot h$ and remark that $h^2\cdot h^2=h^4=3$ and $h^2\cdot S_{42}=9$. The double point formula for $S_{42}\subset X$  (see \cite[Theorem 9.3]{fulton-intersection}) yields $S_{42}^2=41$ and that the restriction of the intersection form to  $\langle h^2, S_{42}\rangle$ has discriminant $3\cdot 41-81=42$.
Let 
$\mathcal V\subset |H^0(\mathcal O_{\p^5}(3))|=\p^{55}$ be the open
set corresponding to smooth cubic hypersurfaces.  
We verified that    $h^0(\mathcal I_{S_{42}}(3))=9$ for a general $[S_{42}]\in\mathcal S_{42}$ and that there exists a smooth
cubic hypersurface through $S_{42}$. Therefore the locus
$$ \mathbf C_{42}=\{([S],[X])\;:\; S\subset X\}\subset\mathcal S_{42}\times\mathcal V,$$ 
has dimension $48+8=56$. The image of $\pi_2:\mathbf C_{42}\to \mathcal V$ has dimension at most 54 because the general cubic fourfold does not contain any surface belonging to $\mathcal S_{42}$.  For every $[X]\in\mathcal \pi_2(\mathbf C_{42})$ we have
$$\dim(\pi_2^{-1}([X]))\geq \dim(\mathbf C_{42})-\dim(\pi_2(\mathbf C_{42}))=56-\dim(\pi_2(\mathbf C_{42}))\geq 56-54=2.$$
Since $h^0(N_{S/X})\geq \dim_{[S]}(\pi_2^{-1}([X]))$ for every $[S]\in \pi_2^{-1}([X])$, where  $\dim_{[S]}(\pi_2^{-1}([X]))$ denotes the dimension
of $\pi_2^{-1}([X])$ at the point $[S]$, to show that a general $[X]\in\mathcal C_{42}$ contains a surface $S_{42}$
it is sufficient to  verify that $h^0(N_{S_{42}/X})=2$ for a fixed $S_{42}$ and for a smooth $X\in|H^0(\mathcal I_{S_{42}}(3))|$, see also  \cite[pp.~284--285]{Nuer} for a similar argument. We verified this  via {\it Macaulay2} in an explicit example and  we can conclude that a general $[X]\in\mathcal C_{42}$  contains a surface $S_{42}$ as above. 
\end{proof}

Since at each step of the algorithm summarized in Remark~\ref{algorithm S42}  we need to introduce only new independent variables,
we deduce that the family $\mathcal S_{42}$ is unirational. 
In other words, our construction yields an explicit dominant rational map
$\PP^{N}\dashrightarrow \mathcal S_{42}$.
As an immediate consequence of this and of Theorem~\ref{C42S42}, we have the following:
\begin{cor}\label{uniraz C42} The irreducible divisor $\mathcal C_{42}$ is  unirational.
\end{cor}

\begin{Remark}\label{rem4.2} 
 A different description of the divisor $\mathcal{C}_{42}$ had been given 
 in \cite{Lai} as the locus of  cubic fourfolds containing 
 a rational scroll of degree $9$ with $8$ nodes. From this Lai deduces that
 $\mathcal C_{42}$ is uniruled. This has been substantially refined in 
 \cite[Theorem 1.1]{FarkasVerraC42}, where the authors prove that the universal K3 surface
 of genus 22 is unirational. This implies the unirationality of the 19-dimensional
 moduli space $\mathcal F_{22}$ of polarized K3 surfaces of genus 22 and hence the unirationality
 of $\mathcal C_{42}$ by a result of Hassett, see for example \cite[Corollary 25]{Levico}. 
\end{Remark}

\subsection{Rationality of cubics fourfolds in \texorpdfstring{$\mathcal C_{42}$}{C42}}\label{ratC42}

The new description of the general cubic fourfolds in $\mathcal C_{42}$ as those containing a general $[S_{42}]\in\mathcal S_{42}$ allows us to deduce the rationality of a general element of $\mathcal C_{42}$ and then, by applying  \cite[Theorem 1]{KontsevichTschinkelInventiones}, of all cubics in $\mathcal C_{42}$.  In the next subsection we shall also put in evidence, via the trisecant flop, the birational connection between  a general pair $(X,S_{42})$ with $S_{42}\subset X$ and its associated K3 surface of degree 42 and of genus 22.

\begin{thm}\label{C42_Flop} Every cubic fourfold in $\mathcal C_{42}$ is rational.
\end{thm}
\begin{proof}Since we have given an algorithm for computing the general surface $S_{42}\subset\PP^{5}$ of the family $\mathcal S_{42}$ 
(see Remark~\ref{algorithm S42}),
we can explicitly construct it using \emph{Macaulay2} and study its geometrical properties.
So let
$\phi:\p^5\map Z\subset\p^{8}$ be the rational map defined 
by the linear system $|H^0(\mathcal I_{S_{42}}(3))|$ of cubic hypersurfaces 
through a general surface $S_{42}\subset\PP^5$ of the family $\mathcal S_{42}$.
One can calculate that the  map $\phi$ is birational onto its image $Z\subset\p^8$, which  has degree 14, sectional genus 15 and ideal generated by 7 cubic forms. Through a general point $q=\phi(p)$ there passes 17 lines contained in $Z$, whose preimages via $\phi$ provide: nine secant lines to $S_{42}$ through $p$, seven 5-secant conics to $S_{42}$ through $p$ and one 8-secant twisted cubic to $S_{42}$ through $p$. Then we  conclude that  $S_{42}$ admits a congruence of 8-secant twisted cubics. Once we have determined the congruence we have also another way of detecting it. Indeed, the linear system $|H^0(\mathcal I^3_{S_{42}}(8))|$ of octic hypersurfaces
with triple points along $S_{42}$
defines a dominant rational map $\mu:\p^5\map W\subset\p^7$ onto a smooth linear section $W$ of $\mathbb G(1,4)\subset\p^9$. The general fiber of $\mu$ is a twisted cubic 8-secant to $S_{42}$ so that the restriction of $\mu$ to a general cubic fourfold $X$ through $S_{42}$ induces a birational map $\mu|_X:X\dashrightarrow W$.  Since
a general cubic fourfold in $|H^0(\mathcal I_{S_{42}}(3))|$ 
through a general surface in $\mathcal S_{42}$ is rational, 
we conclude that 
every cubic fourfold of discriminant $42$ is rational 
by the main result in \cite{KontsevichTschinkelInventiones}.
 We refer to Subsection \ref{contid42} for more details on the above calculations.
\end{proof}

\subsection{Birational model of the associated K3 surface of degree 42 and genus 22 via the trisecant flop}\label{TF42}  A general surface $S_{42}\subset\p^5$ satisfies
Assumption~\ref{assumption1} in Subsection \ref{assdef} and it has the expected trisecant behaviour, see the end of Remark \ref{cs42}. 
 By Theorem \ref{trisflop} the map $\mu$ restricted to a general cubic $X\subset\p^5$ through $S_{42}$
determines a trisecant flop $\tau:X'=\Bl_{S_{42}}X\map W'$ with $W'$ a smooth fourfold. By Theorem \ref{C42_Flop} and by Theorem \ref{contraction}, the congruence of 8-secant twisted cubics to $S_{42}$ induces a birational morphism $\nu:W'\to W$, which is the  blow-up of a  surface $\overline B\subset W\subset \p^7$. 

By studying the birational map $\mu:X\map W$ (see Subsection~\ref{contid42}), we obtain
that $\overline B\subset W\subset\p^7$ is a smooth surface of degree 21 and sectional genus 18. Then $\overline B$ coincides with its support $U$, which is thus a smooth surface of degree 21 and sectional genus 18 with $p_g(U)=1$ and $q(U)=0$.  
We now apply the geometric method developed in Section~\ref{geomet} to deduce that $U$ is the blow-up of a K3 surface at nine distinct points and to describe the linear system giving the embedding.

Considers two surfaces $U_j\subset W$, $j=1,2$ associated via $\mu$ to two general $X_j$'s  through $S_{42}$. 
The one dimensional component $C$ of $U_1\cap U_2$ is a smooth curve of degree 13, of arithmetic genus $-8$ with 9 irreducible components (a general projection of $C$ to $\p^3$ allows a quick verification of all these properties), yielding that 
$C=C_1\cup\cdots\cup C_9$ is the disjoint union of nine smooth rational curves. More precisely, the curve $C\subset\p^7$  consists of five distinct lines, let us say $C_1,\ldots, C_5$, and of four conics $C_6,\ldots, C_9$ (see Subsection \ref{contid42} and the ancillary file \verb!code_section_6.m2!). Since $h^1(N_{C/U_j})=0$, the $C_i$'s are   $(-1)$-curves on $U_j$ for every $i=1,\ldots, 9$ and for $j=1,2$ by Lemma \ref{lemadjunction}. Let $U=U_1$ and let $H\subset U$ be a hyperplane section. Since $H\cdot K_U=-H^2+2g(H)-2=-21+36-2=13=H\cdot C$ we deduce from Lemma \ref{lemadjunction} the existence of the contraction $\pi:U\to U'$ of the  $C_i$'s to 9 distinct points $p_1,\ldots, p_9$ on a smooth K3 surface $U'$. The divisor  $H'=\pi_*(H)\subset U'$ is  ample  and  such that
$(H')^2=21+16+5=42$,  $g(H')=18+4=22$, proving that  
$U'$ is a  K3 surface of degree 42 and genus 22. For a general $X$ through $S_{42}$ we have that 
$U'$ is a  K3 surface of degree 42 and genus 22 with $\rk(\Pic(U'))=1$ (one can argue as  at the end of Subsection \ref{s38} or remark that $U'$ has necessarily 19 moduli equal to the moduli of  a general $X\in\mathcal C_{42}$) so that $|H'|$ gives an embedding $U'\subset\p^{22}$. The linear system $|H|$ on $U$  corresponds to the  hyperplane sections of $U'$ passing through $p_1,\ldots, p_{5}$ and having a point of multiplicity at least two at $p_{6},\ldots, p_9$. In particular, we see that these points do not impose independent conditions to hyperplane sections of $U'\subset\p^{22}$. 

By considering the map associated to the linear system $|H+C_1+\cdots+ C_5+2(C_6+\cdots+C_9)|$ on $U$,   we construct its image $U'\subset\p^{22}$, which is thus a general K3 surface of degree 42 and genus 22, see the ancillary file \verb!code_section_6.m2!. From this one can deduce an alternative proof of one of the main results in \cite{FarkasVerraC42}, according to which the moduli space of polarized K3 surfaces of degree 42 and genus 22 is unirational.

\section{Explicit examples of trisecant flops in \emph{Macaulay2}}\label{contiM2}

We mostly  
used 
the computer algebra system
\emph{Macaulay2} \cite{macaulay2}
with the packages 
\href{https://faculty.math.illinois.edu/Macaulay2/doc/Macaulay2/share/doc/Macaulay2/Cremona/html/index.html}{\emph{Cremona}} \cite{packageCremona} and 
\href{https://faculty.math.illinois.edu/Macaulay2/doc/Macaulay2/share/doc/Macaulay2/SpecialFanoFourfolds/html/index.html}{\emph{SpecialFanoFourfolds}} \cite{SpecialFanoFourfoldsSource} 
to study surfaces in $\mathbb{P}^5$ admitting congruences of $(3e-1)$-secant curves
of degree $e$ and the rational maps given by hypersurfaces of degree $3e-1$ 
having points of multiplicity $e$ along these surfaces. 
We refer to the documentations of these two packages for technical computational details.
In particular, the first one provides tools for working with rational maps, 
such as computing their fibers, 
checking birationality and determining inverse maps.
The validity of 
these computations relies on  the fact that 
the irreducible components $\mathcal{S}_d$ 
of the Hilbert schemes considered here 
are 
\emph{explicitly unirational}.
This means that we have a procedure to 
determine the equations of the generic member of $\mathcal{S}_d$
as a function of a 
number of specific independent variables.
Therefore, 
by 
adding more variables we can also take the generic point of $\mathbb{P}^5$
and compute, for instance, the generic fiber 
of the map defined by the cubics through 
the generic $[S_d]\in\mathcal S_d$, 
which will depend on all these variables. 
In practice this is far beyond 
what computers can do today. 
Anyway,
the answer we get is equivalent to the one obtained 
on the original field via a generic specialization 
of the variables 
and, above all, 
the generic specialization 
commutes with this type of computation. 

\subsection{Rationality of cubic fourfolds in $\mathcal C_{38}$}\label{contid38}
Here, we consider a specific example related to Subsection~\ref{s38} (row (iii) of Table~\ref{tabella}).

In the following  code, we 
 produce a 
surface $S=S_{38}\subset\PP^5$ 
obtained as the image
of $\p^2$ by the linear system of plane curves of degree 10 having 10 randomly-chosen triple points.
We work over the finite field $K=\mathbb{F}_{10000019}$ for speed reasons. 
{\footnotesize
\begin{Verbatim}[commandchars=&!$]
Macaulay2, version 1.19
&colore!darkorange$!i1 :$ &colore!airforceblue$!needsPackage$ "&colore!airforceblue$!SpecialFanoFourfolds$"; &colore!Sepia$!-- v2.5$
&colore!darkorange$!i2 :$ K = &colore!darkspringgreen$!ZZ$/10000019;
&colore!darkorange$!i3 :$ S = &colore!airforceblue$!surface$({10,0,0,10},K);
o3 : ProjectiveVariety, surface in PP^5
\end{Verbatim}
} \noindent 
We now compute the rational map $\mu$ defined by the linear system of 
quintic hypersurfaces of $\PP^5$ which are singular along $S$.\footnote{More generally,
the command \texttt{rationalMap(S,d,e)} returns 
the rational map defined by
a basis of the linear system $|H^0(\mathcal I_{S}^e(d))|$
of hypersurfaces of degree $d$ having points of multiplicity $e$ along $S$.} 
From the information obtained by its projective degrees 
we deduce
that $\mu$ is a dominant rational map onto $\PP^4$
with generic fibre of dimension $1$ and degree $2$ and with
base locus of dimension $3$ and degree $5^2-19=6$.
{\footnotesize
\begin{Verbatim}[commandchars=&\[\]]
&colore[darkorange][i4 :] mu = &colore[airforceblue][rationalMap](S,5,2);
o4 : RationalMap (rational map from PP^5 to PP^4)
&colore[darkorange][i5 :] &colore[airforceblue][projectiveDegrees] mu
o5 = {1, 5, 19, 13, 2, 0}
\end{Verbatim}
} \noindent 
Next we compute a special random fibre $F$ of the map $\mu$.
{\footnotesize
\begin{Verbatim}[commandchars=&\[\]]
&colore[darkorange][i6 :] p = &colore[airforceblue][point source] mu;  &colore[Sepia][-- a random point on P^5]
&colore[darkorange][i7 :] F = mu^* mu p;             
\end{Verbatim}
} \noindent 
It easy to verify directly that $F$ is an irreducible $5$-secant conic to $S$ passing through $p$. 
Letting $\varphi:\PP^5\dashrightarrow\PP^9$ 
the rational map defined by the linear system of cubics through $S$,
one can also see that 
$F$ 
coincides with the pull-back $\overline{\varphi^{-1}(L)}$ 
of the unique line $L\subset \overline{\varphi(\PP^5)}\subset\PP^9$ 
 passing through $\varphi(p)$ that is not the image 
 of a secant line to $S$ passing through $p$ (see \cite[Section~5]{RS1} for details on this computation).
 Finally, the following lines of code tell us
that the restriction $\mu'$ of $\mu$ to a randomly-chosen cubic fourfold $X$ containing $S$ is a birational map 
whose inverse map is defined by forms of degree $9$ and whose base locus scheme
has dimension $2$ and degree $9^2 - 27 = 54$.
{\footnotesize
\begin{Verbatim}[commandchars=&\[\]]
&colore[darkorange][i8 :] X = &colore[airforceblue][random](3,S);
o8 : ProjectiveVariety, hypersurface in PP^5
&colore[darkorange][i9 :] mu' = mu|X;
o9 : RationalMap (rational map from X to PP^4)
&colore[darkorange][i10 :] &colore[airforceblue][projectiveDegrees] mu'
o10 = {3, 15, 27, 9, 1}
\end{Verbatim}
} \noindent 
The smooth surface $U\subset\PP^4$ of degree $12$ and sectional genus $14$ 
determining the inverse map of $\mu':X\dashrightarrow\PP^4$ 
can be calculated as the non-reduced part of the base locus scheme of $\mu'^{-1}$.
Alternatively, one can use the same method described in the next subsection.
The full code to determine $U$, the exceptional curves on $U$, and the map $U\dashrightarrow U'\subset\PP^{20}$
onto the K3 surface $U'\subset\PP^{20}$ of degree $38$ and genus $20$
is included in the ancillary file \verb!code_section_6.m2!.
We omit it here for the sake of brevity.

\subsection{Rationality of cubic fourfolds in $\mathcal C_{42}$}\label{contid42}
Here, we perform similar calculations as above, but considering a specific example related to Subsection~\ref{sez C42 nuova} (row (0) of Table~\ref{tabella}).

Using the algorithm given in Remark~\ref{algorithm S42}, 
one can calculate the homogeneous ideal of a randomly-chosen
surface $S=S_{42}$ in the
$48$-dimensional family $\mathcal S_{42}$
constructed in Proposition~\ref{S42V}.
This has been implemented in the \emph{Macaulay2} package \emph{SpecialFanoFourfolds}. So, 
to get the ideal of such a surface $S$ and of a randomly-chosen (smooth) cubic fourfold $X$ containing it, 
it is enough to run the following code:
{\footnotesize
\begin{Verbatim}[commandchars=&!$]
&colore!darkorange$!i11 :$ X = &colore!airforceblue$!specialCubicFourfold$("general cubic 4-fold of discriminant 42",K);
o11 : ProjectiveVariety, cubic fourfold containing a surface of degree 9 and sectional genus 2
&colore!darkorange$!i12 :$ S = &colore!airforceblue$!ideal surface$ X;
\end{Verbatim}
} \noindent 
The following is one of the ways to
compute relatively quickly 
the rational map $\mu$ defined by the linear system of 
octic hypersurfaces of $\PP^5$ having triple points along $S$. This calculation takes about one minute.
{\footnotesize
\begin{Verbatim}[commandchars=&!$]
&colore!darkorange$!i13 :$ mu = &colore!airforceblue$!rationalMap$(S^3 : &colore!airforceblue$!first gens ring$ S,8);
o13 : RationalMap (rational map from PP^5 to PP^7)
\end{Verbatim}
} \noindent 
Now, with the same code used above, we compute a special random fibre $F$ of the map $\mu$.
{\footnotesize
\begin{Verbatim}[commandchars=&!$]
&colore!darkorange$!i14 :$ p = &colore!airforceblue$!point source$ mu;
&colore!darkorange$!i15 :$ F = mu^* mu p;
\end{Verbatim}
} \noindent 
Here is a practical way to get that $F$ is a twisted cubic curve which is $8$-secant to $S$.
{\footnotesize
\begin{Verbatim}[commandchars=&!$]
&colore!darkorange$!i16 :$ ? F
o16 = smooth cubic curve of genus 0 in PP^5 cut out by 5 hypersurfaces of degrees (1,1,2,2,2)
&colore!darkorange$!i17 :$ ? (F + S)
o17 = 0-dimensional subscheme of degree 8 in PP^5
\end{Verbatim}
} \noindent 
The code above also tells us that the (closure of the) image of $\mu$ is a subvariety 
$W\subset\PP^7$ of dimension $4$. To get the equations of $W$, one can use the command \texttt{image mu},
but this takes a while. 
A faster way is to calculate the scheme of quadrics containing $W$,
as shown below. Then, 
since one verifies easily that it is a smooth connected fourfold,
we deduce that $W$ coincides with this fourfold.
{\footnotesize
\begin{Verbatim}[commandchars=&!$]
&colore!darkorange$!i18 :$ W = &colore!airforceblue$!image$(2,mu);
\end{Verbatim}
} \noindent 
The surface $U\subset W\subset\PP^7$
determining the inverse map of
the restriction of $\mu$ to the cubic fourfold $X$
can be find without computing the inverse map. Indeed,
the intersection of $W$ with a general cubic hypersurface through $U$ 
is given by the image $\overline{\mu(X\cap X')}\subset W$, where $X'$ 
is a general cubic hypersurface through $S$.
For efficiency, we suggest to calculate 
$\overline{\mu(X\cap X')}$ by interpolating 
the images via $\mu$ of several points on $X\cap X'$.
The following is a possible implementation. It 
gives the homogeneous ideal of $U\subset\PP^7$ and takes about 5 minutes.
{\footnotesize
\begin{Verbatim}[commandchars=&!$]
&colore!darkorange$!i19 :$ U = &colore!airforceblue$!trim sum$(8,j->(X' = &colore!airforceblue$!random$(3,surface X);
               &colore!airforceblue$!ideal take$((&colore!airforceblue$!intersect apply$(90,i->mu &colore!airforceblue$!ideal point$(X * X')))_*,6)));
\end{Verbatim}
} \noindent 
The exceptional curves in the surface $U$ can be determined by taking 
another randomly-chosen cubic fourfold $\tilde{X}$ through $S$ (for instance, as in line \texttt{i8}), and then 
by calculating the corresponding surface $\tilde{U}$ (for instance, by re-executing the line \texttt{i19}).
One expects that  the exceptional curves in $\tilde{U}$ are the same as those in $U$ (see Section~\ref{geomet}),
so one tries to determine them as the top-dimensional components of the intersection $U\cap\tilde{U}$ 
(this can be done quickly after taking generic projections in $\PP^3$ and $\PP^2$). 
In the cases under consideration, the one-dimensional
components of this
intersection are 9 disjoint curves:
5 lines $C_1,\ldots,C_5$ and 4 conics $C_6,\ldots,C_9$, which are all the 
$(-1)$-curves on $U$ (see also Subsection \ref{TF42}).
Having determined these curves, one can also calculate 
the map $f:U\to U'\subset\PP^{22}$ onto the K3 surface $U'\subset\PP^{22}$ 
of degree $42$ and genus $22$,  given by the linear system $|H+\sum_{i=1}^5 C_i + 2 \sum_{i=1}^{4} C_{5+i}|$,
where $H$ denotes the hyperplane section of $U$.
For the sake of brevity, the full code to get the map $f$ and the equations of its image is omitted  here and it is included in the ancillary file \verb!code_section_6.m2! together with all the lines of code given in this section. 
Note also that this kind of calculations
can be automated using the function
\href{https://faculty.math.illinois.edu/Macaulay2/doc/Macaulay2/share/doc/Macaulay2/SpecialFanoFourfolds/html/_associated__K3surface_lp__Special__Cubic__Fourfold_rp.html}{\texttt{associatedK3surface}}
provided by the package \emph{SpecialFanoFourfolds}.

\subsection{Other worked out examples of trisecant flops}\label{contiTables} 
We provide a \emph{Macaulay2}  package  named \emph{TrisecantFlops},\footnote{It is available at \url{https://github.com/giovannistagliano/TrisecantFlops}.}
which produces explicit examples of trisecant flops in accordance to Table~\ref{tabella} in the next section.
This package can be downloaded automatically from \emph{SpecialFanoFourfolds}.
So, by typing \texttt{trisecantFlop i} (where $i$ is an integer between $0$ and  $17$), 
will build up a birational map $\mu: X\dashrightarrow W$ as in the $i$-th row of Table~\ref{tabella}.
For instance, we now consider the 
third 
example.
{\footnotesize
\begin{Verbatim}[commandchars=&!$]
&colore!darkorange$!i20 :$ mu = &colore!airforceblue$!trisecantFlop$ 3;
o20 : RationalMap (birational map from cubic fourfold containing 
      a surface of degree 10 and sectional genus 6 to PP^4)
&colore!darkorange$!i21 :$ &colore!airforceblue$!projectiveDegrees inverse$ mu
o21 = {1, 9, 27, 15, 3}
\end{Verbatim}
} \noindent 
We can obtain the smooth surface $S\subset\mathbb{P}^5$ of degree $10$ and sectional genus $6$ by giving the following command. 
{\footnotesize
\begin{Verbatim}[commandchars=&!$] 
&colore!darkorange$!i22 :$ S = &colore!airforceblue$!surface source$ mu;
o22 : ProjectiveVariety, surface in PP^5
&colore!darkorange$!i23 :$ (degree S, sectionalGenus S)
o23 = (10, 6)
\end{Verbatim}
} \noindent 
Analogous, the non--minimal K3 surface $U\subset\mathbb{P}^4$ is obtained as follows.
{\footnotesize
\begin{Verbatim}[commandchars=&!$] 
&colore!darkorange$!i24 :$ U = &colore!airforceblue$!surface target$ mu;
o24 : ProjectiveVariety, surface in PP^4
&colore!darkorange$!i25 :$ (degree U, sectionalGenus U)
o25 = (12, 14)
\end{Verbatim}
} \noindent 
Finally, the following command yields an extension to $\PP^5$ of the map $\mu:X\dashrightarrow W=\PP^4$ 
whose general fibre is a $5$-secant conic to the surface $S$.
{\footnotesize
\begin{Verbatim}[commandchars=&!$] 
&colore!darkorange$!i26 :$ &colore!airforceblue$!extend$ mu; 
o26 : RationalMap (dominant rational map from PP^5 to PP^4)
\end{Verbatim}
}

\section{Summary table of examples of trisecant flops}\label{sez Tabelle}

We provide in Table~\ref{tabella} a list of 18 examples of
  maps $\mu: X\dashrightarrow W$ as in diagram \eqref{diagramma1}, where
 $X$ is a cubic fourfold in $\mathcal C_d$ 
 which contains a surface $S\subset\PP^5$ 
 admitting a congruence of $(3e-1)$-secant rational  curves of degree $e$. There
 are some different behaviours:
 \begin{itemize}
 \item $[X]\in\mathcal C_d$ is general except in  ({xi}) and ({xii}); 
 \item $S\subset\mathbb{P}^5$ is cut out 
by cubics except in  ({xii}), ({xiii}), ({xiv}), and ({xvii});  
\item the map $X\dashrightarrow Y$ defined by the cubics through $S$ is birational except in ({xv}); 
\item the cubics through $S$ satisfy the condition $\mathcal K_3$ except  
in ({0}), ({v}), ({viii}),  ({x})--({xvii}).
\end{itemize}
In Table~\ref{tabella3}, we 
give some additional information on 
the examples of surfaces $S\subset X\subset\PP^5$ considered in Table~\ref{tabella}.
Most of this was achieved using  
the functions \href{https://faculty.math.illinois.edu/Macaulay2/doc/Macaulay2/share/doc/Macaulay2/SpecialFanoFourfolds/html/_detect__Congruence_lp__Special__Cubic__Fourfold_cm__Z__Z_rp.html}{\texttt{detectCongruence}} and \href{https://faculty.math.illinois.edu/Macaulay2/doc/Macaulay2/share/doc/Macaulay2/SpecialFanoFourfolds/html/_parameter__Count.html}{\texttt{parameterCount}}
from the package \emph{SpecialFanoFourfolds}, see also Subsection \ref{contiTables}. 

When the corresponding surface $U\subset W$ is smooth, the proof that it is a (non-minimal) K3 surfaces follows the paths used in Section~\ref{geomet} (see also Subsection~\ref{TF42}) by determining explicitly the exceptional curves on $U$. Once these exceptional curves are determined, one finds the description of the linear system on $U$ in terms of the hyperplane sections of the K3 surface $U'$. 
When $U\subset W$ is singular, one first determines the exceptional curves as above; then takes a {linear normalization} to obtain a smooth surface and finally, if necessary, follows the previous path.

\begin{table}[htbp]
\centering 
\tabcolsep=0.7pt 
\begin{adjustbox}{width=\textwidth}
\begin{tabular}{|c|cccccc|}
\hline
& \hspace{3pt}$d$\hspace{3pt} & \hspace{3pt}$e$\hspace{3pt} & {$S\subset X\subset \PP^5$} & {$W$} & {$U\subset W$} & $Y$ \\
\hline \hline 
0 & $42$ & $3$ & \normalsize{\begin{tabular}{c} Rational surface of degree $9$ and sectional \\ genus $2$ with $5$ nodes, which is a special \\ projection of the image of $\PP^2$ in $\PP^8$  via \\ the linear system  of quartic curves with \\$3$  simple points and one double point  \end{tabular}} & $\GG(1,4)\cap\PP^7\subset\PP^7$ & \normalsize{\begin{tabular}{c} Non-minimal K3 surface of degree $21$ \\and sectional genus $18$, cut out in $\PP^{7}$ \\ by  $5$ quadrics  and $8$ cubics \end{tabular}} & \normalsize{\begin{tabular}{c} $4$-fold of degree $14$ in $\PP^{7}$ \\ cut out by $7$ cubics \end{tabular}} \\
\hline 
\rowcolor{gray!5.0}
i & $14$ & $2$ & \normalsize{\begin{tabular}{c} Isomorphic projection of a smooth surface \\in $\p^6$ of degree $8$ and sectional genus $3$,\\ obtained as the image of $\p^2$ via the linear \\system of quartic curves with $8$ general \\ base points \end{tabular}} & $\PP^4$ & \normalsize{\begin{tabular}{c} Singular $K3$ surface of degree $10$ \\and   sectional genus $7$, cut out by $12$ \\quintics  and having $8$  singular points \end{tabular}} & \normalsize{\begin{tabular}{c} $4$-fold of degree $28$ in $\PP^{11}$ \\ cut out by $16$ quadrics \end{tabular}} \\
\hline
ii & ${26}$ & $2$ & \normalsize{\begin{tabular}{c} Rational scroll of degree $7$ with $3$ nodes  \end{tabular}} & $\PP^4$ & \normalsize{\begin{tabular}{c} Singular K3 surface of degree $10$ \\ and sectional genus $8$, cut out \\ by $12$ quintics and one sextic, \\ and having $3$ singular points  \end{tabular}} & \normalsize{\begin{tabular}{c} $4$-fold of degree $29$ in $\PP^{11}$ \\ cut out by $15$ quadrics \end{tabular}}\\
\hline
\rowcolor{gray!5.0}
iii & ${38}$ & $2$ & \normalsize{\begin{tabular}{c} Smooth surface of degree $10$ and sectional \\ genus $6$, obtained as the image of $\PP^2$ via \\ the linear system  of curves of degree $10$\\  with $10$ general triple points \end{tabular}} & $\PP^4$ & \normalsize{\begin{tabular}{c} Smooth non-minimal K3 surface of\\ degree $12$ and sectional genus $14$ \\ cut out by $9$ quintics \end{tabular}} & \normalsize{\begin{tabular}{c} $4$-fold of degree $20$ in $\PP^{8}$ \\ cut out by $16$ cubics \end{tabular}} \\
\hline
iv & ${26}$ & $2$ & \normalsize{\begin{tabular}{c} Projection of a smooth del Pezzo surface \\ of degree $7$ in $\PP^7$  from a line intersecting \\ the secant variety in one general point \end{tabular}} & $\GG(1,4)\cap\PP^7\subset\PP^7$ & \normalsize{\begin{tabular}{c} Non-minimal K3 surface of degree $17$ \\and sectional genus $11$, cut out in $\PP^{7}$ \\ by  $5$ quadrics  and $13$ cubics \end{tabular}} & \normalsize{\begin{tabular}{c} $4$-fold of degree $34$ in $\PP^{12}$ \\ cut out by $20$ quadrics \end{tabular}} \\
\hline
\rowcolor{gray!5.0}
v & ${38}$ & $3$ & \normalsize{\begin{tabular}{c} Rational scroll of degree $8$ with $6$ nodes  \end{tabular}} & $\GG(1,5)\cap\PP^{10}\subset\PP^{10}$ & \normalsize{\begin{tabular}{c} Smooth non-minimal K3 surface \\ of degree $22$ and sectional genus $14$, \\ cut out in $\PP^{10}$ by  $24$ quadrics \end{tabular}} & \normalsize{\begin{tabular}{c} $4$-fold of degree $17$ in $\PP^{8}$ \\ cut out by $3$ quadrics and \\$4$ cubics \end{tabular}}\\
\hline
vi & ${14}$ & $3$ & \normalsize{\begin{tabular}{c} Projection from $3$ general internal \\points of a minimal $K3$ surface of \\ degree $14$ and sectional genus $8$ \end{tabular}} & \normalsize{\begin{tabular}{c} Cubic fourfold \end{tabular}} & \normalsize{\begin{tabular}{c} Projection from $3$ general internal \\ points of a minimal $K3$ surface of \\ degree $14$ and sectional genus $8$ \end{tabular}} & \normalsize{\begin{tabular}{c} Complete intersection in $\PP^{7}$ \\ of $2$ quadrics and one cubic \end{tabular}} \\
\hline
\rowcolor{gray!5.0}
vii & ${14}$ &  $3$ & \normalsize{\begin{tabular}{c} Projection of a $K3$ surface of degree $10$ \\ and sectional genus $6$ in $\PP^6$ from a \\ general point on its secant variety \end{tabular}} & \normalsize{\begin{tabular}{c} Gushel-Mukai fourfold in $\PP^8$ \end{tabular}} & \normalsize{\begin{tabular}{c} Smooth minimal K3 surface of \\ degree $14$ and sectional genus $8$ \end{tabular}} & \normalsize{\begin{tabular}{c} Hypercubic section of a \\hyperplane  section of $\mathbb{G}(1,4)$  \end{tabular}}\\
\hline
viii & ${14}$ &  $5$ & \normalsize{\begin{tabular}{c} General hyperplane section of a conic \\ bundle  in $\PP^6$ of degree $13$ and \\ sectional genus $12$  \end{tabular}} & \normalsize{\begin{tabular}{c} Complete intersection of \\ three quadrics in $\PP^7$ \end{tabular}} & \normalsize{\begin{tabular}{c} Smooth non-minimal K3 surface \\of  degree $13$ and sectional genus $8$, \\ cut out by $9$ quadrics \end{tabular}} & \normalsize{\begin{tabular}{c} Hypersurface of degree $5$ in $\PP^5$ \end{tabular}}\\
\hline
\rowcolor{gray!5.0}
ix & ${14}$ & $5$ & \normalsize{\begin{tabular}{c} General hyperplane section of a pfaffian \\ threefold  in $\PP^6$ of degree $14$ and \\ sectional genus $15$  \end{tabular}} & $\mathbb{G}(1,6)\cap \PP^{14}\subset\PP^{14}$ & \normalsize{\begin{tabular}{c} Smooth minimal K3 surface of \\ degree $14$ and sectional genus $8$ \\ embedded in $\PP^8\subset\PP^{14}$\end{tabular}} & \normalsize{\begin{tabular}{c} Hypersurface of degree $5$ in $\PP^5$ \end{tabular}} \\
\hline
x & ${38}$ & $5$ & \normalsize{\begin{tabular}{c} Smooth surface of degree $11$ and sectional \\ genus $7$, obtained as the image of $\PP^2$ via \\ the linear system  of curves of degree $12$ with \\ one general simple point, $4$ general triple \\ points, and $6$ general quadruple points  \end{tabular}} & $\GG(1,5)\cap\PP^{10}\subset\PP^{10}$ & \normalsize{\begin{tabular}{c} Smooth non-minimal K3 surface \\ of degree $25$ and sectional  genus \\ $17$,  cut out in $\PP^{10}$ by  $21$ quadrics \end{tabular}} & \normalsize{\begin{tabular}{c} Hypersurface of degree $7$ in $\PP^5$ \end{tabular}}\\
\hline
\rowcolor{gray!5.0}
xi & $38$ & $3$ & \normalsize{\begin{tabular}{c} Projection of an octic del Pezzo surface \\ isomorphic to $\mathbb{F}_1$  from a plane  intersecting \\ the secant variety in $3$  general points  \end{tabular}} & $\mathbb{G}(1,3)\subset\mathbb{P}^5$ & \normalsize{\begin{tabular}{c} Non-minimal $K3$ surface of degree $13$ \\and   sectional genus $10$, cut out in $\mathbb{P}^5$ by \\ one quadric, $9$ quartics, and $3$ quintics \end{tabular}} & \normalsize{\begin{tabular}{c} $4$-fold of degree $17$ in $\mathbb{P}^{8}$ \\ cut out by $3$ quadrics and \\  $4$ cubics \end{tabular}} \\
\hline 
xii & $38$ & $3$ & \normalsize{\begin{tabular}{c} Projection of an octic del Pezzo surface \\ isomorphic to $\mathbb{F}_0$  from a plane  intersecting \\ the secant variety in $3$  general points \\ (cut out by $10$ cubics and one quartic) \end{tabular}} & $\mathrm{LG}_3(\mathbb{C}^6)\cap\mathbb{P}^{11}\subset\mathbb{P}^{11}$ & \normalsize{\begin{tabular}{c} Non-minimal $K3$ surface  of degree $26$ \\ and   sectional genus $17$,  cut out in $\mathbb{P}^{11}$ \\ by $30$ quadrics \end{tabular}} & \normalsize{\begin{tabular}{c} $4$-fold of degree $18$ in $\mathbb{P}^{8}$ \\ cut out by $2$ quadrics and \\ $8$ cubics \end{tabular}} \\
\hline 
\rowcolor{gray!5.0}
xiii & $14$ & $3$ & \normalsize{\begin{tabular}{c} Isomorphic projection of a smooth surface \\in $\mathbb{P}^7$ of degree $8$ and sectional genus $2$,\\ obtained as the image of $\mathbb{P}^2$ via the linear \\system of quartic curves with $4$ simple \\ base points and one double point \\ (cut out by $10$ cubics and $3$ quartics) \end{tabular}} & \normalsize{\begin{tabular}{c} Complete intersection \\ of $2$ quadrics in $\mathbb{P}^6$ \end{tabular}} & \normalsize{\begin{tabular}{c} Singular $K3$ surface of degree $14$ \\and   sectional genus $8$, cut out \\ in $\mathbb{P}^{6}$ by $2$ quadrics and $9$ cubics, \\ and having one singular point \end{tabular}} & \normalsize{\begin{tabular}{c} Complete intersection \\ of $4$ quadrics in $\mathbb{P}^8$ \end{tabular}} \\
\hline 
xiv & $26$ & $5$ & \normalsize{\begin{tabular}{c}  Rational scroll of degree $8$ with $4$ nodes \\ (cut out by $8$ cubics and $3$ quartics) \end{tabular}} & $\mathbb{G}(1,3)\subset\mathbb{P}^5$ & \normalsize{\begin{tabular}{c} Non-minimal $K3$ surface of degree $14$ \\and   sectional genus $11$, cut out in $\mathbb{P}^5$ by \\ one quadric, $7$ quartics, and $2$ quintics \end{tabular}} & \normalsize{\begin{tabular}{c} Complete intersection in $\mathbb{P}^6$ \\of a quadric and a quartic \end{tabular}} \\
\hline
\rowcolor{gray!5.0}
xv & $26$ & $5$ & \normalsize{\begin{tabular}{c} Surface of degree $13$ and sectional genus $11$ \\ cut out by $6$ cubics and with an ordinary  node, \\ which is obtained as a special projection of \\ a minimal K3 surface of degree $26$ of genus $14$ \end{tabular}} & Cubic fourfold & \normalsize{\begin{tabular}{c} A surface of the same kind as $S$ \end{tabular}} & \normalsize{\begin{tabular}{c} -- \end{tabular}} \\
\hline 
xvi & $26$ & $6$ & \normalsize{\begin{tabular}{c} Surface of degree $11$ and sectional genus $6$ \\ cut out by $7$ cubics and with $3$ non-normal nodes, \\ which is obtained as a special projection of a \\ smooth surface of degree $11$ and sec. genus $6$ in $\PP^6$ \end{tabular}} & \normalsize{\begin{tabular}{c} $\mathbb{S}^{10}\cap\mathbb{P}^{9}\subset\mathbb{P}^{9}$, where \\ $\mathbb{S}^{10}\subset\mathbb{P}^{15}$  is the spinorial \\ variety \end{tabular}} & \normalsize{\begin{tabular}{c} Non-minimal $K3$ surface  of degree $21$ \\ and   sectional genus $14$,  cut out in $\mathbb{P}^{9}$ \\ by $16$ quadrics and one cubic \end{tabular}} & \normalsize{\begin{tabular}{c} Hypersurface of degree $5$ in $\mathbb{P}^{5}$ \end{tabular}} \\
\hline 
\rowcolor{gray!5.0}
xvii & $26$ & $5$ & \normalsize{\begin{tabular}{c} Smooth surface of degree $11$ and sectional \\ genus $7$, obtained as the image of $\PP^2$ via \\ the linear system of curves of degree $8$ with \\ $3$ simple base points,  $8$ general double \\ points, and $2$ general triple points \\ (cut out by $7$ cubics and one quartic) \end{tabular}} & \normalsize{\begin{tabular}{c} $\mathbb{G}(1,3)\subset\PP^5$ \end{tabular}} & \normalsize{\begin{tabular}{c} Singular K3 surface of degree $15$ \\ and sectional genus $12$, cut out \\ in $\mathbb{P}^5$ by one quadric and $6$ quartics, \\ and having $9$ singular points \end{tabular}} & \normalsize{\begin{tabular}{c} Hypersurface of degree $6$ in $\mathbb{P}^5$ \end{tabular}} \\
\hline
\end{tabular} 
\end{adjustbox}
 \caption{\scriptsize{Examples of maps $\mu: X\dashrightarrow W$ as in diagramm \eqref{diagramma1}, where
 $[X]\in \mathcal C_d$ and $S\subset X$
 admits a congruence of $(3e-1)$-secant rational  curves of degree $e$.}}
\label{tabella} 
\end{table}

\begin{table}[htbp]
\centering 
\tabcolsep=1.5pt 
\begin{adjustbox}{width=\textwidth}
\begin{tabular}{|c|ccccc|ccc|}
\hline
\rowcolor{gray!5.0}
     & \footnotesize{\begin{tabular}{c} $2$-secant lines to $S$ \\ passing through $p$ \end{tabular}} & \footnotesize{\begin{tabular}{c} $5$-secant conics to $S$ \\ passing through $p$ \end{tabular}} & \footnotesize{\begin{tabular}{c} $8$-secant cubics to $S$ \\ passing through $p$\end{tabular}} & \footnotesize{\begin{tabular}{c} $11$-secant quartics to $S$ \\ passing through $p$ \end{tabular}} & \footnotesize{\begin{tabular}{c} $14$-secant quintics to $S$ \\ passing through $p$ \end{tabular}} & $h^0(\mathcal{I}_{S/\PP^5}(3))$ & $h^0(N_{S/\PP^5})$ & $h^0(N_{S/X})$ \\
\hline 
0 & $9$ & $7$ & $1$ & $0$ & $0$ & $9$ & $48$ & $2$ \\
\rowcolor{gray!5.0}
i  & $7$  & $1$ &  $0$ & $0$ & $0$ & $13$ & $49$ & $7$ \\
ii  & $7$  & $1$ &  $0$ & $0$ & $0$ & $13$ & $44$ & $2$ \\
\rowcolor{gray!5.0}
iii  & $7$  & $1$ &  $0$ & $0$ & $0$ & $10$ & $47$ & $2$ \\
iv  & $5$  & $1$ &  $0$  & $0$ & $0$ & $14$ & $42$ & $1$ \\
\rowcolor{gray!5.0}
v  & $9$  & $4$ &  $1$ & $0$ & $0$ & $10$ & $47$ & $2$   \\
vi & $13$  & $10$ &  $1$ & $0$ & $0$ & $9$ & $60$ & $14$ \\
\rowcolor{gray!5.0}
vii & $11$  & $6$ &  $1$ & $0$ & $0$ & $10$ & $59$ & $14$ \\
viii & $19$ & $44$ & $48$ & $8$ & $1$ & $7$ & $68$ & $20$ \\
\rowcolor{gray!5.0}
ix & $21$ & $56$ & $42$ & $0$ & $1$ & $7$ & $77$ & $29$ \\
x & $11$ & $16$ & $16$ & $4$ & $1$ & $7$ & $49$ & $1$ \\
\rowcolor{gray!5.0}
xi  & $7$  & $6$ &  $1$ & $0$ & $0$ & $10$ & $44$ & $0$  \\
xii  & $7$  & $4$ &  $1$ & $0$ & $0$ & $10$ & $44$ & $0$  \\
\rowcolor{gray!5.0}
xiii  & $9$  & $6$ &  $1$ & $0$ & $0$ & $10$ & $49$ & $4$ \\
xiv & $11$ & $12$ & $16$ & $8$ & $1$ & $8$ & $49$ & $2$ \\
\rowcolor{gray!5.0}
xv  & $19$  & -- &  -- & -- & $1$ & $6$ & $58$ & $9$ \\
xvi  & $15$  & $31$ &  $44$ & $24$ & $5$ & $7$ & $56$ & $8$ \\
\rowcolor{gray!5.0}
xvii  & $13$  & $22$ & $26$ & $10$ & $1$ & $7$ & $53$ & $5$ \\
\hline
\end{tabular}
\end{adjustbox}
 \caption{\scriptsize{Surfaces $S$ contained in a cubic fourfold $X\subset\PP^5$ as in Table~\ref{tabella}; $p$ is a general point of $\PP^5$.}} 
\label{tabella3} 
\end{table}

\providecommand{\bysame}{\leavevmode\hbox to3em{\hrulefill}\thinspace}
\providecommand{\MR}{\relax\ifhmode\unskip\space\fi MR }
\providecommand{\MRhref}[2]{%
  \href{http://www.ams.org/mathscinet-getitem?mr=#1}{#2}
}
\providecommand{\href}[2]{#2}


\begin{thebibliography}{10}

\bibitem{AR}
A.~Alzati and F.~Russo, \emph{Some elementary extremal contractions between
  smooth varieties arising from projective geometry}, Proc. Lond. Math. Soc.
  \textbf{89} (2004), 25--53.
  
\bibitem{Ando}
T.~Ando, \emph{On extremal rays of higher dimensional varieties},
  Invent. Math. \textbf{81} (1985), no.~2, 347--357.
  

\bibitem{Artin}
M.~Artin, \emph{Algebraization of formal moduli: {II}. {E}xistence of
  modifications}, Ann. of Math. \textbf{91} (1970), no.~1, 88--135.

\bibitem{Bauer}
I.~Bauer, \emph{The classification of surfaces in $\mathbb{P}^5$ having few
  trisecant lines}, Rend. Sem. Mat. Univ. Pol. Torino \textbf{56} (1998),
  1--20.

\bibitem{BRS}
M.~Bolognesi, F.~Russo, and G.~Staglian{\`o}, \emph{Some loci of rational cubic
  fourfolds}, Math. Ann. \textbf{373} (2019), no.~1, 165--190.

\bibitem{CC}
L.~Chiantini and C.~Ciliberto, \emph{A few remarks on the lifting problem},
  Ast\'erisque \textbf{218} (1993), 95--109.

\bibitem{DES}
W.~Decker, L.~Ein, and F.-O. Schreyer, \emph{Construction of surfaces in
  $\mathbb{P}^4$}, J. Algebraic Geom. \textbf{2} (1993), no.~2, 185--237.

\bibitem{Fano}
G.~Fano, \emph{Sulle forme cubiche dello spazio a cinque dimensioni contenenti
  rigate razionali del $4^\circ$ ordine}, Comment. Math. Helv. \textbf{15}
  (1943), no.~1, 71--80.

\bibitem{FarkasVerraC42}
G.~Farkas and A.~Verra, \emph{The unirationality of the moduli space of {K3}
  surfaces of degree 42},  Math. Ann. \textbf{380} (2021), 953--973.  
\bibitem{FS}
C.~Fontanari and E.~Sernesi, \emph{Non-surjective {G}aussian maps for singular
  curves on {K3} surfaces}, Collect. Math. \textbf{70} (2019), no.~1, 107--115.

\bibitem{FNakano}
A.~Fujiki and S.~Nakano, \emph{Supplement to ``on the inverse of monoidal
  transformation''}, Publ. RIMS Kyoto Univ. \textbf{7} (1971), 637--644.

\bibitem{fulton-intersection}
W.~Fulton, \emph{Intersection theory}, Ergeb. Math. Grenzgeb. (3), no.~2,
  Springer-Verlag, 1984.

\bibitem{macaulay2}
D.~R. Grayson and M.~E. Stillman, \emph{{\sc Macaulay2} --- {A} software system
  for research in algebraic geometry (version 1.19)}, Home page:
  \url{http://www.math.uiuc.edu/Macaulay2/}, 2021.

\bibitem{GrusonPeskine}
L.~Gruson and C.~Peskine, \emph{On the smooth locus of aligned {H}ilbert
  schemes, the $k$-secant lemma and the general projection theorem}, Duke Math.
  J. \textbf{162} (2013), no.~3, 553--578.

\bibitem{HMK}
C.~D. Hacon and J.~Mc Kernan, \emph{The {S}arkisov program}, J. Algebraic Geom.
  \textbf{22} (2013), 389--405.

\bibitem{Levico}
B.~Hassett, \emph{Cubic fourfolds, {K3} surfaces, and rationality questions}, in
  {Rationality Problems in Algebraic Geometry}, Levico Terme, Italy 2015, Springer International Publishing, Cham,
  2016, pp.~29--66.

\bibitem{HPT}
B.~Hassett, A.~Pirutka, and Y.~Tschinkel, \emph{Stable rationality of quadric
  surface bundles over surfaces}, Acta Math. \textbf{220} (2018), no.~2,
  341--365.

\bibitem{HS}
M.~Hoff and G.~Staglian{\`o}, \emph{New examples of rational {G}ushel-{M}ukai
  fourfolds}, Math. Z. \textbf{296} (2020), 1585--1591.

\bibitem{HKS}
K.~Hulek, S.~Katz, and F.-O. Schreyer, \emph{Cremona transformations and
  syzygies}, Math. Z. \textbf{209} (1992), no.~1, 419--443.

\bibitem{KollarHyp}
J.~Koll\'{a}r, \emph{Algebraic hypersurfaces}, Bull. Amer. Math. Soc.
  \textbf{56} (2019), 543--568.

\bibitem{KontsevichTschinkelInventiones}
M.~Kontsevich and Y.~Tschinkel, \emph{Specialization of birational types},
  Invent. Math. \textbf{217} (2019), no.~2, 415--432.

\bibitem{kuz2}
A.~Kuznetsov, \emph{Derived categories view on rationality problems},
in  {\it Rationality Problems in Algebraic Geometry}, Levico Terme, Italy 2015, Springer International Publishing, Cham,
  2016, pp.~67--104.
  
\bibitem{IL}   A.~Iliev, L.~Manivel, \emph{Fano manifolds of degree ten and $EPW$ sextics}, 
Ann. Sci. Ecole Norm. Sup.  \textbf{44} (2011), 393--426.

\bibitem{Lai}
K.~Lai, \emph{New cubic fourfolds with odd-degree unirational
  parametrizations}, Algebra \& Number Theory \textbf{11} (2017), 1597--1626.

\bibitem{LLW}
Y.-P. Lee, H.-W. Lin, and C.-L. Wang, \emph{Flops, motives, and invariance of
  quantum rings}, Ann. of Math. \textbf{172} (2010), no.~1, 243--290.

\bibitem{Mori}
S.~Mori, \emph{Threefolds whose canonical bundles are not numerically
  effective}, Ann. of Math. \textbf{116} (1982), no.~1, 133--176.

\bibitem{Nakano}
S.~Nakano, \emph{On the inverse of monoidal transformation}, Publ. {RIMS} Kyoto
  Univ. \textbf{6} (1971), no.~3, 483--502.

\bibitem{NOtte}
J.~Nicaise and J.~C. Ottem, \emph{Tropical degenerations and stable
  rationality}, preprint: \url{https://arxiv.org/abs/1911.06138}, 2019.

\bibitem{Nuer}
H.~Nuer, \emph{Unirationality of moduli spaces of special cubic fourfolds and
  {K3} surfaces}, Algebr. Geom. \textbf{4} (2015), 281--289.

\bibitem{Ran}
Z.~Ran, \emph{Unobstructedness of filling secants and the {G}ruson--{P}eskine
  general projection theorem}, Duke Math. J. \textbf{164} (2015), no.~4,
  697--722.

\bibitem{Rogora}
E.~Rogora, \emph{On projective varieties for which a family of multisecant
  lines has dimension larger than expected}, preprint n. 28/96, Dip. di
  Matematica, Universit\` a degli {S}tudi di {R}oma \emph{{La Sapienza}},
  available at \url{http://www1.mat.uniroma1.it/people/rogora/pdf/28.pdf}.

\bibitem{OADP}
F.~Russo, \emph{On a theorem of {S}everi}, Math. Ann. \textbf{316} (2000),
  no.~1, 1--17.
  
  \bibitem{Umi}  \bysame, {\it On the Geometry of Some Special Projective Varieties}, Lecture Notes of the Unione Matematica Italiana, Springer Verlag, Berlin, 2016.

\bibitem{Explicit}
F.~Russo and G.~Staglian\`o, \emph{Explicit Rationality of Some Special Fano 
  Fourfolds}, in  {\it Rationality of
  Varieties}, Progress in Mathematics \textbf {342}, pp.~323--343, Birkh\" auser, Cham, 2021.

\bibitem{RS1}
\bysame, \emph{Congruences of $5$-secant conics and the rationality of some
  admissible cubic fourfolds}, Duke Math. J. \textbf{168} (2019), no.~5,
  849--865.


\bibitem{Schreieder2019}
S.~Schreieder, \emph{Stably irrational hypersurfaces of small slopes}, J. Amer.
  Math. Soc. \textbf{32} (2019), no.~4, 1171--1199.

\bibitem{Semple}
J.~G. Semple, \emph{On representations of the ${S}_k$'s of ${S}_n$ and of the
  {G}rassmann manifolds ${G}(k,n)$}, Proc. Lond. Math. Soc. \textbf{s2-32}
  (1931), no.~1, 200--221.

\bibitem{ST}
J.~G. Semple and J.~A. Tyrrell, \emph{The ${T}_{2,4}$ of ${S}_6$ defined by a
  rational surface $^3{F}^8$}, Proc. Lond. Math. Soc. \textbf{s3-20} (1970),
  205--221.

\bibitem{SpecialFanoFourfoldsSource}
G.~Staglian\`o, \emph{{\emph{SpecialFanoFourfolds}}: a {\sc macaulay2} package
  for working with special cubic fourfolds and special {G}ushel-{M}ukai
  fourfolds, version~2.5}, source code available at
  \url{https://github.com/Macaulay2/M2/tree/master/M2/Macaulay2/packages}.

\bibitem{packageCremona}
\bysame, \emph{a {Macaulay2} package for computations with rational maps}, J.
  Softw. Alg. Geom. \textbf{8} (2018), no.~1, 61--70.

\bibitem{StaDP5}
\bysame, \emph{On some families of Gushel-Mukai fourfolds}, 
 preprint: \url{https://arxiv.org/abs/2002.07026},
 to appear in Algebra \& Number Theory,
  2020.
  
\bibitem{TaVaA}
S.~Tanimoto and A.~V\'arilly-Alvarado, \emph{Kodaira dimension of moduli of
  special cubic fourfolds}, J. Reine Angew. Math. \textbf{752} (2019),
  265--300.
  
 \bibitem{Todd} J.~A.~Todd, \emph{The Locus Representing the Lines of Four-Dimensional Space and its Application to Linear Complexes in Four Dimensions}, Proc. London Math. Soc. \textbf{30} (1930), 513--550. 
 
 
\bibitem{Totaro2015}
B.~Totaro, \emph{Hypersurfaces that are not stably rational}, J. Amer. Math.
  Soc. \textbf{29} (2015), no.~3, 883--891.

\bibitem{Vermeire}
P.~Vermeire, \emph{Some results on secant varieties leading to a geometric flip
  construction}, Compos. Math. \textbf{125} (2001), no.~3, 263--282.

\bibitem{Verra} A.~Verra, \emph{The unirationality of the moduli spaces of genus 14 or lower},
Compositio Math. \textbf{141} (2005), 1425--1444. 

\bibitem{Voisin}
C.~Voisin, \emph{Segre classes of tautological bundles on {H}ilbert schemes of
  surfaces}, Algebr. Geom. \textbf{6} (2019), no.~2, 186--195.

\end{thebibliography}
\end{document}